\documentclass[nosumlimits,twoside]{amsart}
\usepackage{amsfonts, amsmath, amssymb}
\usepackage{graphicx}
\usepackage{float}
\usepackage{srcltx}
\usepackage[all]{xy}
\usepackage{version}
\usepackage[final]{showlabels}
\usepackage[T1]{fontenc}
\usepackage{xcolor}
\usepackage{enumerate}
\usepackage[normalem]{ulem}
\usepackage{bm}
\usepackage{latexsym}
\usepackage[2emode]{psfrag}
\usepackage{yhmath}
\usepackage{array}
\usepackage{dsfont}
\usepackage{mathrsfs}
\usepackage{tikz-cd}
\usepackage{comment}
\usepackage[colorinlistoftodos,prependcaption,textsize=tiny]{todonotes}

\newtheorem{theorem}{\sc Theorem}[section]
\newtheorem{proposition}[theorem]{\sc Proposition}

\newtheorem{lemma}[theorem]{\sc Lemma}
\newtheorem{corollary}[theorem]{\sc Corollary}
\theoremstyle{definition}
\newtheorem{definition}[theorem]{\sc Definition}

\theoremstyle{remark}
\newtheorem{remark}[theorem]{\sc Remark}

\newenvironment{invisible}{{\noindent\sc \colorbox{yellow}{Invisible:}\;}\color{gray}}{\medskip}
\excludeversion{invisible}
\allowdisplaybreaks
\excludeversion{invisible?}
\excludeversion{proof?}

\newcommand{\Cc}{\mathcal{C}}

\newcommand{\Hc}{\mathrm{Hopf}_{\mathrm{coc}}(\mathrm{Vec}_{G})}

\begin{document}
\title[Commutators and crossed modules of color Hopf algebras]{Commutators and crossed modules of color Hopf algebras}
\author{Andrea Sciandra}
\address{
\parbox[b]{\linewidth}{University of Turin, Department of Mathematics ``G. Peano'', via
Carlo Alberto 10, I-10123 Torino, Italy}}
\email{andrea.sciandra@unito.it}


\maketitle

\begin{abstract}
In a previous paper we showed that the category of cocommutative color Hopf algebras is semi-abelian in case the group $G$ is abelian and finitely generated and the characteristic of the base field is different from 2 (not needed if $G$ is finite of odd cardinality). 
Here we describe the commutator of cocommutative color Hopf algebras and we explain the Hall's criterion for nilpotence and the Zassenhaus Lemma. Furthermore, we introduce the category of color Hopf crossed modules and we explicitly show that this is equivalent to the category of internal crossed modules in the category of cocommutative color Hopf algebras and to the category of simplicial cocommutative color Hopf algebras with Moore complex of length 1.  
\end{abstract}

\tableofcontents

\maketitle

\section{Introduction}
The notion of semi-abelian category was introduced in \cite{Semi-ab} in order to give a categorical “generalization” of the category of groups, as abelian categories “generalize” abelian groups. Among the classical examples of semi-abelian categories there are the categories of groups, Lie algebras and associative rings.

In \cite[Theorem 6.1]{As} we showed that the category $\Hc$ of cocommutative color Hopf algebras is semi-abelian if the abelian group $G$ is finitely generated and the characteristic of the base field $\Bbbk$ is different from 2 (not needed if $G$ is finite of odd cardinality). This result generalizes \cite[Theorem 2.10]{MG} given for the category of ordinary cocommutative Hopf algebras over a field $\Bbbk$ of arbitrary characteristic, which can be seen as a special case of our result by considering $G$ equal to the trivial group. Moreover, taking $G=\mathbb{Z}_{2}$ and $\mathrm{char}\Bbbk\not=2$, we obtain that the category of cocommutative super Hopf algebras, extensively used in Mathematics and Physics, is semi-abelian.

Semi-abelian categories provide a good categorical framework to develop an approach to commutator theory and they present natural notions of semi-direct product \cite{BouJan}, internal action \cite{BoJa} and crossed module \cite{Ja}. Moreover, some famous theorems for groups have categorical generalizations in a semi-abelian category, like the Hall's criterion for nilpotence \cite{Gray} and the Zassenhaus Lemma \cite{OlSt}. Semi-abelian categories are also suitable to study (co)homology of non-abelian structures \cite{EV2} and the Moore complex structure \cite{bou2}, \cite{EV}. Therefore, we will study some of these properties for the category $\Hc$, under the previous assumptions on $G$ and $\Bbbk$. \medskip

The structure of the paper is the following. First, for the convenience of the reader, we include preliminaries about the categorical-algebraic properties which we will explore for the category $\Hc$. In Section 3 we give an explicit description of commutators in the category $\Hc$, using the action $\xi$ defined in \cite{As}. Then we explain the Hall's criterion for nilpotence and the Zassenhaus Lemma for the category $\Hc$ in Section 4. In Section 5 we revise the equivalence between internal actions and split extensions in $\Hc$ and, by introducing the notion of color Hopf crossed module, we show the equivalence between the categories of color Hopf crossed modules and internal crossed modules in $\Hc$. Finally, in Section 6, we study the equivalence between the categories of color Hopf crossed modules and simplicial cocommutative color Hopf algebras with Moore complex of length one. Observe that generalizations of the equivalences in Sections 5 and 6 can be found in the very general results given in \cite{Bohm} and \cite{Bohm2}, but here we give an explicit description of these in reference to the category $\Hc$ which can also be used to investigate further other topics like crossed squares and 2-crossed modules of cocommutative color Hopf algebras. \medskip

\noindent\textit{Notations and conventions}. Sometimes the identity morphisms in a category $\Cc$ are denoted by 1. If the category $\Cc$ is pointed, the zero object is denoted by $\mathbf{0}$ and the zero morphism by 0. Given a morphism $f:A\to B$ in $\Cc$, the kernel and the cokernel of $f$, if they exist, are denoted by $\mathrm{ker}(f):\mathrm{Ker}(f)\to A$ and $\mathrm{coker}(f):B\to\mathrm{Coker}(f)$ or $\mathrm{coker}(f):B\to B/A$. 
All vector spaces are understood to be over the field $\Bbbk$ and by linear maps we mean $\Bbbk$-linear maps. The unadorned tensor product $\otimes$ stands for $\otimes_{\Bbbk}$. All linear maps whose domain is a tensor product will usually be defined on generators and understood to be extended by linearity. Algebras over $\Bbbk$ will be associative and unital.

\section{Preliminaries}
In this section we recall some categorical-algebraic properties, which we will explore for the category $\Hc$. 
A category $\Cc$ is \textit{semi-abelian} if it is pointed, finitely cocomplete, (Barr)-exact and protomodular \cite{Semi-ab}. Thus, $\Cc$ has zero object, all finite colimits and limits, so that the protomodularity is equivalent to the fact that the Split Short Five Lemma holds in $\Cc$ (see \cite[Proposition 3.1.2]{BorBou}). Moreover, every morphism in $\Cc$ factorizes as a regular epimorphism (i.e., a coequalizer of a pair of morphisms in $\Cc$) followed by a monomorphism and this factorization is stable under pullbacks (i.e., $\Cc$ is regular) and any equivalence relation is the kernel pair of a morphism in $\Cc$. We refer the reader to \cite{BorBou} for many details about semi-abelian categories. 
\\

\noindent\textbf{The Huq commutator}. 
In any pointed category $\Cc$ with binary products, one says that two subobjects $i:X\to A$ and $j:Y\to A$ of the same object $A$ \textit{commute} 
(in the sense of Huq \cite{Huq}) if there exists an arrow $p$ making the following diagram commute:
\begin{equation}\label{commutator}
    \begin{tikzcd}
	X & X\times Y & Y \\
	& A
	\arrow[from=1-1, to=1-2, "\langle{1,0}\rangle"]
	\arrow[from=1-3, to=1-2, "\langle{0,1}\rangle"']
	\arrow[from=1-1, to=2-2, "i"']
	\arrow[from=1-3, to=2-2, "j"]
	\arrow[dashed, from=1-2, to=2-2, "p"]
\end{tikzcd}
\end{equation}
If the category $\Cc$ is protomodular, such an arrow $p$ is unique, when it exists \cite{BoGr}. In a protomodular category $\Cc$ with finite limits, a subobject $u:S\to A$ of $A$ is called \textit{normal} if it is normal to an equivalence relation on $A$ (see \cite[Definitions 3.2.1 and 3.2.9]{BorBou}) and, if $\Cc$ is pointed, exact and protomodular, then $u$ is normal if and only if it is the kernel of a morphism in $\Cc$, as it is shown in \cite[Proposition 3.2.20]{BorBou}. In particular, the latter holds true for a semi-abelian category. In a semi-abelian category $\Cc$, the \textit{Huq commutator} of two normal subobjects $i:X\to A$ and $j:Y\to A$ is the smallest normal subobject $g:B\to A$ of $A$ such that its cokernel $q:=\mathrm{coker}(g):A\to A/B$ has the property that the images $q(X)$ and $q(Y)$ by $q$ commute in the quotient as it is shown in the following diagram:
\[
\begin{tikzcd}
	&&&& q(X) & q(X)\times q(Y) & q(Y) \\
	& X && Y \\
	B && A &&& A/B
	\arrow[from=3-1, to=3-3, "g"']
	\arrow[from=2-2, to=3-3, "i"]
	\arrow[from=2-4, to=3-3, "j"']
	\arrow[from=3-3, to=3-6, "q"']
	\arrow[from=1-5, to=1-6]
	\arrow[from=1-7, to=1-6]
	\arrow[dashed, from=1-6, to=3-6]
	\arrow[from=1-7, to=3-6]
	\arrow[from=1-5, to=3-6]
\end{tikzcd}
\]
Usually it is denoted by $[X,Y]_{\mathrm{Huq}}\to A$.
Recall also that an object $A$ in $\Cc$ is called \textit{nilpotent} \cite{Huq} if there exists a non-negative integer $n$ such that $\gamma^{n}_{A}(A)=0$ where, denoted by $\mathrm{Norm}(A)$ the class of normal subobjects of $A$, the map $\gamma_{A}:\mathrm{Norm}(A)\to\mathrm{Norm}(A)$ sends a normal subobject $X$ of $A$
to $[A,X]_{\mathrm{Huq}}$. 
The least such $n$ is the nilpotency class of $A$. 

If a semi-abelian category $\Cc$ is also algebraically coherent \cite{CGV} then a generalization of the Hall's criterion for nilpotence holds:

\begin{theorem}(cf. \cite[Theorem 3.4]{Gray})\label{nilpotence}
Let $\Cc$ be an algebraically coherent semi-abelian category and $p:E\to B$ 
a regular epimorphism  where $B$ is a nilpotent object in $\Cc$. If the kernel of $p$ is contained in the Huq commutator $[N,N]_{\mathrm{Huq}}$ of a nilpotent normal subobject $N$ of $E$, then $E$ is nilpotent. Furthermore, if $N$ is of nilpotency class $c$ and $B$ is of nilpotency class $d$, then $E$ is of nilpotency class at most $\frac{c(c+1)}{2}(d-1)+c$.
\end{theorem}

\noindent\textbf{Semi-direct products}. Recall that, given a morphism $v:W\to Y$ in a semi-abelian category $\Cc$, the inverse image functor $v^{*}:\mathrm{Pt}_{Y}(\Cc)\to\mathrm{Pt}_{W}(\Cc)$ of the fibration of points is monadic, see \cite{BouJan} and \cite[Theorem 5.1.13]{BorBou}. From \cite{BouJan} and \cite[Definition 5.2.8]{BorBou} also recall that, given a semi-abelian category $\Cc$ and an object $G$ in $\Cc$, a $G$-\textit{algebra} is an algebra for the monad $\mathbf{T}_{G}$ corresponding to the monadic functor $\alpha_{G}^{*}:\mathrm{Pt}_{G}(\Cc)\to\Cc$, where $\alpha_{G}:\mathbf{0}\to G$ is the unique arrow from the zero object $\mathbf{0}$ of $\Cc$ to $G$. Given a $G$-algebra $(X,\xi)$, the \textit{semi-direct product} $(X,\xi)\rtimes G$ of $(X,\xi)$ and $G$ is the domain part $H$ of the point \begin{tikzcd}
	H & G
	\arrow[shift left=1, from=1-1, to=1-2, "p"]
	\arrow[shift left=1, from=1-2, to=1-1,"s"]
\end{tikzcd}
corresponding to $(X,\xi)$ through the equivalence $\mathrm{Pt}_{G}(\Cc)\cong \Cc^{\mathbf{T}_{G}}$. \\

\noindent\textbf{Reflexive multiplicative graphs and internal groupoids}. Recall from \cite{Ca} that a \textit{reflexive-multiplicative graph} in a category with pullbacks $\Cc$ is a diagram 
\begin{equation}\label{RMG}
\begin{tikzcd}
	A_{1}\times_{A_{0}}A_{1} & A_{1} && A_{0}
	\arrow[from=1-1, to=1-2, "m"]
	\arrow[shift right=2.5, from=1-2, to=1-4, "\gamma"']
	\arrow[shift left=2.5, from=1-2, to=1-4, "\delta"]
	\arrow["i"{description}, from=1-4, to=1-2]
\end{tikzcd}
\end{equation}
where $(A_{1}\times_{A_{0}}A_{1},\pi_1,\pi_2)$ is the pullback of the pair $(\delta,\gamma)$ in $\Cc$, such that $\delta\circ i=\gamma\circ i=\mathrm{Id}_{A_{0}}$ (i.e., it is a reflexive graph) and $m\circ(\mathrm{Id}_{A_{1}},i\circ\delta)=\mathrm{Id}_{A_{1}}=m\circ(i\circ\gamma,\mathrm{Id}_{A_{1}})$, where $(\mathrm{Id}_{A_{1}},i\circ\delta):A_{1}\to A_{1}\times_{A_{0}}A_{1}$ and $(i\circ\gamma,\mathrm{Id}_{A_{1}}):A_{1}\to A_{1}\times_{A_{0}}A_{1}$ are induced by the universal property of the pullback $A_{1}\times_{A_{0}}A_{1}$. Usually $\delta$ is called the  \textit{domain morphism} or \textit{source morphism}, $\gamma$ the  \textit{codomain morphism} or \textit{target morphism}, $i$ the \textit{identity morphism} and $m$ the \textit{multiplication}. 

A morphism of reflexive-multiplicative graphs is given by a pair of morphisms $(f_{1}:A'_{1}\to A_{1},f_{0}:A'_{0}\to A_{0})$ in $\Cc$ such that the four squares of corresponding arrows of the following diagram commute: 
\[
\begin{tikzcd}
	A'_{1}\times_{A'_{0}}A'_{1} & A'_{1} && A'_{0} \\
	A_{1}\times_{A_{0}}A_{1} & A_{1} && A_{0}
	\arrow[from=1-1, to=1-2, "m'"]
	\arrow[from=1-1, to=2-1, "f_{1}\times f_{1}"']
	\arrow[from=2-1, to=2-2, "m"']
	\arrow[shift right=2.5, from=2-2, to=2-4, "\gamma"']
	\arrow[shift left=2.5, from=2-2, to=2-4,"\delta"]
	\arrow["i"{description}, from=2-4, to=2-2]
	\arrow[shift right=2.5, from=1-2, to=1-4, "\gamma'"']
	\arrow[shift left=2.5, from=1-2, to=1-4,"\delta'"]
	\arrow["i'"{description}, from=1-4, to=1-2]
	\arrow[from=1-2, to=2-2, "f_{1}"]
	\arrow[from=1-4, to=2-4, "f_{0}"]
\end{tikzcd}
\]
A reflexive-multiplicative graph is called an \textit{internal category} if the multiplication $m$ further satisfies 
\[
\delta\circ m=\delta\circ\pi_{2},\ \ \gamma\circ m=\gamma\circ\pi_{1}\ \ \text{and}\ \ m\circ(\mathrm{Id}_{A_{1}}\times_{A_{0}}m)=m\circ(m\times_{A_{0}}\mathrm{Id}_{A_{1}}).
\]
Moreover, an internal category is called \textit{internal groupoid} if there exists a morphism $\iota:A_{1}\to A_{1}$ such that
\[
\delta\circ\iota=\gamma,\ \ \gamma\circ\iota=\delta,\ \ m\circ(\iota,\mathrm{Id}_{A_{1}})=i\circ \delta\ \ \text{and}\ \ m\circ(\mathrm{Id}_{A_{1}},\iota)=i\circ\gamma.
\]
We denote by RMG$(\Cc)$ the category of reflexive-multiplicative graphs in $\Cc$ with morphisms of reflexive-multiplicative graphs and by Grpd$(\Cc)$ the category of internal groupoids in $\Cc$, with morphisms of reflexive-multiplicative graphs. It is shown in \cite{Ca} that if $\Cc$ is Mal'tsev then the forgetful functor $F:\mathrm{Grpd}(\Cc)\to\mathrm{RMG}(\Cc)$ is an isomorphism of categories and the groupoid structure on a reflexive graph is unique whenever it exists. In \cite{Ja} G. Janelidze introduced the category XMod$(\Cc)$ of internal crossed modules in a semi-abelian category $\Cc$, equivalent to the category Grpd$(\Cc)$. \\

For any two equivalence relations over the same object of a category $\Cc$, there is a notion of connector \cite{BoGr} and, when such a connector exists, one says that the two equivalence relations \textit{centralize} (in the sense of Smith).
If two equivalence relations centralize in the sense of Smith then their normalizations commute in the sense of Huq, while the converse is not true in general (see \cite{BoGr}). If a category $\Cc$ is such that two equivalence relations centralize in the sense of Smith if and only if their normalizations commute in the sense of Huq, then $\Cc$ is said to satisfy the so-called condition (SH)
\cite{BoGr}. Note that any action representable category satisfies the 
condition (SH), since this is true more generally for any action accessible category \cite{BouJan2}. Moreover, a reflexive graph in a Mal'tsev category $\Cc$
\begin{equation}\label{refgraph}
    \begin{tikzcd}
	A_{1} && A_{0}
	\arrow[shift left=3, from=1-1, to=1-3, "\delta"]
	\arrow[shift right=3, from=1-1, to=1-3, "\gamma"']
	\arrow["i"{description}, from=1-3, to=1-1]
\end{tikzcd}
\end{equation}
has an internal structure of groupoid if and only if $\mathrm{Eq}(\delta)$ and $\mathrm{Eq}(\gamma)$, the kernel pairs of $\delta$ and $\gamma$, centralize to each other (in the sense of Smith) \cite{Ca}. All the previous observations lead us to the following result (see also \cite[Remark 5.3]{MG}):

\begin{proposition}
\label{RMGGrpd}
In a pointed, Mal'tsev category $\Cc$ which satisfies the 
condition (SH), the following statements are equivalent for the reflexive graph \eqref{refgraph}:
\begin{enumerate}
    \item it is a reflexive-multiplicative graph;
    \medskip
    \item it is an internal category;
    \medskip
    \item it is an internal groupoid;
    \medskip
    \item  $\mathrm{Eq}(\delta)$ and $\mathrm{Eq}(\gamma)$ centralize to each other (in the sense of Smith);
    \medskip
    \item it satisfies the equality: $[\mathrm{Ker}(\delta),\mathrm{Ker}(\gamma)]_{\mathrm{Huq}}=\mathbf{0}$. 
\end{enumerate}
\end{proposition}

\noindent\textbf{The Zassenhaus Lemma}. Recall from \cite{JaZ} that a category is \textit{normal} if it is pointed regular and every regular epimorphism is a normal epimorphism. A normal category is \textit{ideal determined} if it has binary coproducts and the normal image of a normal monomorphism is again a normal monomorphism \cite{JMT}. In particular, any semi-abelian category is normal and ideal determined (see \cite{Semi-ab}). In \cite[Proposition 5.2]{Wy}, O. Wyler stated the Zassenhaus Lemma for an ideal determined category $\Cc$, using the notion of \textit{asymmetric join}: given a kernel $k:K\to A$ in $\Cc$ with cokernel $f:A\to A/K$ and a monomorphism $m:M\to A$, the asymmetric join is the subobject  $k\vee_{A}m:K\vee_{A}M\to A$ of $A$ such that $(K\vee_{A}M,k\vee_{A}m,\pi_{2})$ is the pullback of the pair $(f,f\circ m)$. In \cite[Lemma 2.8]{OlSt} it is shown that, if $\Cc$ is semi-abelian, the asymmetric join $K\vee_{A}M$ coincides with the supremum of $K$ and $M$ (as subobjects of $A$), i.e., the smallest subobject of $A$ containing $K$ and $M$, which always exists in a category with binary coproducts and a factorization of every morphism as a strong epimorphism followed by a monomorphsim. Thus, from \cite[Theorem 3.3]{OlSt} and \cite[Corollary 3.5]{OlSt}, we recall the Zassenhaus Lemma for a semi-abelian category: 

\begin{theorem}\label{Zassenhaus}
Let $\Cc$ be a semi-abelian category, $K\to U$ and $L\to V$ two kernels in $\Cc$ and $U\to A$ and $V\to A$ two monomorphisms in $\Cc$. Then we have isomorphisms
\[
\frac{K\vee_{U}(U\cap V)}{K\vee_{U}(U\cap L)}\cong\frac{U\cap V}{(K\cap V)\vee_{U\cap V}(L\cap U)}\cong\frac{L\vee_{V}(U\cap V)}{L\vee_{V}(K\cap V)}.
\]
\end{theorem}

\noindent\textbf{Simplicial objects and Moore complexes}. First recall that, given a pointed category $\Cc$, a \textit{chain complex} $(X_{\bullet},\partial_{\bullet})$ in $\Cc$ is given by a collection of objects $(X_{n})_{n\in\mathbb{Z}}$ in $\Cc$ and a collection of morphisms $(\partial_{n}:X_{n}\to X_{n-1})_{n\in\mathbb{Z}}$ in $\Cc$, called the differentials, such that $\partial_{n}\circ\partial_{n+1}=0$, for all $n\in\mathbb{Z}$. A morphism in a pointed, regular and protomodular category $\Cc$ is called \textit{proper} if its image is a kernel \cite{bou}. A chain complex is proper whenever all its differentials are \cite{EV}. We denote the category of chain complexes in $\Cc$ by Ch($\Cc$). \medskip

The \textit{simplicial category} $\Delta$ has objects given by finite ordinals $[n]=\{0,\ldots,n\}$, with $n\in\mathbb{N}$, and morphisms $[n]\to[m]$ given by order preserving maps from $\{0,\ldots,n\}$ to $\{0,\ldots,m\}$. 
Given a category $\Cc$, the category Simp$(\Cc)$ of simplicial objects and simplicial morphisms of $\Cc$ is the functor category Fun$(\Delta^{\mathrm{op}},\Cc)$ \cite{May}. Thus, a \textit{simplicial object} $X:\Delta^{\mathrm{op}}\to\Cc$ in a category $\Cc$
is given by the following data: a collection of objects $(X_{n})_{n\in\mathbb{N}}$ in $\Cc$, morphisms in $\Cc$
\[
d^{n}_{i}:X_{n}\to X_{n-1},\ \text{for}\ 0\leq i\leq n\ \ \text{and}\ \ s^{n+1}_{j}:X_{n}\to X_{n+1},\ \text{for}\ 0\leq j\leq n
\]
called the \textit{face operators} and the \textit{degeneracy operators}, respectively, subject to the following \textit{simplicial identities}:
\begin{enumerate}
\item[1)] $d^{n-1}_{i}\circ d^{n}_{j}=d^{n-1}_{j-1}\circ d^{n}_{i}$ if $i<j$; \medskip
\item[2)] $s^{n+1}_{i}\circ s^{n}_{j}=s^{n+1}_{j+1}\circ s^{n}_{i}$ if $i\leq j$; \medskip
\item[3)] $d^{n}_{i}\circ s^{n}_{j}=s^{n-1}_{j-1}\circ d^{n-1}_{i}$ if $i<j$,\ $d^{n}_{j}\circ s^{n}_{j}=d^{n}_{j+1}\circ s^{n}_{j}=\mathrm{Id}$,\ $d^{n}_{i}\circ s^{n}_{j}=s^{n-1}_{j}\circ d^{n-1}_{i-1}$ if $i>j+1$
\end{enumerate}
for all $n\in\mathbb{N}$. 
\medskip

Given $n\in\mathbb{N}$, a $n$-\textit{truncated} simplicial object in a category $\Cc$ is a functor $X:\Delta_{n}^{\mathrm{op}}\to\Cc$, where $\Delta_{n}$ is the full subcategory of $\Delta$ whose objects are the naturals $\leq n$. We denote the category of $n$-truncated simplicial objects in $\Cc$ by Simp$_{n}(\Cc)$. For each $n\in\mathbb{N}$ there is a \textit{truncation} functor tr$_{n}:\mathrm{Simp}(\Cc)\to\mathrm{Simp}_{n}(\Cc)$ which simply forgets the objects $X_{i}$ of a simplicial object $X$ for dimensions higher than $n$. If $\Cc$ has finite limits then tr$_{n}$ admits a right adjoint cosk$_{n}$ called the $n$-\textit{coskeleton} functor, while if $\Cc$ has finite colimits then tr$_{n}$ admits a left adjoint sk$_{n}$, called the $n$-\textit{skeleton} functor, see \cite{Du}.

Let us quickly recall the construction of the $n$-coskeleton functor for a finitely complete category $\Cc$. Given a $n$-truncated simplicial object $X$, consider the face operators $d^{n}_{0},\ldots,d^{n}_{n}:X_{n}\to X_{n-1}$, then $X_{n+1}$ is defined as an object in $\Cc$ with morphisms $f_{0},\ldots,f_{n+1}:X_{n+1}\to X_{n}$ such that $d^{n}_{i}\circ f_{j}=d^{n}_{j-1}\circ f_{i}$ for all $i<j$ which is universal with this property: given morphisms $\pi_{0},\ldots,\pi_{n+1}:A\to X_{n}$ in $\Cc$ such that $d^{n}_{i}\circ\pi_{j}=d^{n}_{j-1}\circ\pi_{i}$ for all $i<j$ then there is a unique morphism $\alpha:A\to X_{n+1}$ such that $f_{i}\circ\alpha=\pi_{i}$:
\[
\begin{tikzcd}
	A & X_{n} \\
	X_{n+1} & X_{n} & X_{n-1} & {}
	\arrow[dotted, no head, from=2-3, to=2-4]
	\arrow[from=1-1, to=2-1,"\alpha"']
	\arrow[from=1-2, to=2-2, "\mathrm{Id}"]
	\arrow[shift left=3, from=1-1, to=1-2, "\pi_{n+1}"]
	\arrow[shift right=3, from=1-1, to=1-2,"\pi_{0}"]
	\arrow[shift left=3, from=2-1, to=2-2, "f_{n+1}"]
	\arrow[shift right=3, from=2-1, to=2-2, "f_{0}"]
	\arrow[shift left=3, from=2-2, to=2-3, "d^{n}_{n}"]
	\arrow[shift right=3, from=2-2, to=2-3, "d^{n}_{0}"]
\end{tikzcd}
\]
Furthermore, the universal property of $X_{n+1}$ allows to define degeneracy operators $s^{n+1}_{i}:X_{n}\to X_{n+1}$ for $0\leq i\leq n$.
This yields an $(n+1)$-truncated simplicial objects in $\Cc$ and, proceeding inductively, one defines the $n$-coskeleton functor. \medskip

Given a simplicial object $X$ in a pointed category $\Cc$ with pullbacks, 
the \textit{Moore chain complex} $(M(X)_{\bullet},\partial_{\bullet})$ is the chain complex defined by:
\begin{enumerate}
    \item $M(X)_{n}=\mathbf{0}$ for $n<0$ and $M(X)_{0}=X_{0}$; \medskip
    \item $M(X)_{n}=\bigcap_{i=0}^{n-1}{\mathrm{Ker}(d^{n}_{i})}$ for $n\geq1$; \medskip
    \item $\partial_{n}=d^{n}_{n}\circ\cap_{i=0}^{n-1}{\mathrm{ker}(d^{n}_{i})}:M(X)_{n}\to M(X)_{n-1}$ for $n\geq1$ (and the zero morphism for $n\leq0$),
\end{enumerate}
  see \cite[Definition 3.1]{EV}. Hence there is a functor $M:\mathrm{Simp}(\Cc)\to\mathrm{Ch}(\Cc)$ and, if $\Cc$ is pointed and protomodular, then $M$ is conservative, i.e., it reflects isomorphisms \cite{bou2}. In \cite[Theorem 3.6]{EV} it is shown that, given a pointed, exact and protomodular category $\Cc$ and a simplicial object $X$ in $\Cc$, then $(M(X)_{\bullet},\partial_{\bullet})$ is a proper chain complex of $\Cc$. Moreover, we recall the following result:

\begin{theorem}(cf. \cite[Theorem 5.3]{Caf})\label{thm:lengthMoore}
Let $\Cc$ be a pointed category with finite limits. For a simplicial object $X$ with corresponding Moore complex $M(X)$, then the Moore complex of $\mathrm{cosk}_{n}\mathrm{tr}_{n}(X)$ satisfies:
\begin{enumerate}
    \item $M(\mathrm{cosk}_{n}\mathrm{tr}_{n}(X))_{i}=M(X)_{i}$ for $i\leq n$; \medskip
    \item $M(\mathrm{cosk}_{n}\mathrm{tr}_{n}(X))_{n+1}=\mathrm{Ker}(\partial_{n}:M(X)_{n}\to M(X)_{n-1})$; \medskip
    \item  $M(\mathrm{cosk}_{n}\mathrm{tr}_{n}(X))_{i}=\mathbf{0}$ for $i>n+1$.
\end{enumerate}
\end{theorem}

As a consequence, given a $n$-truncated simplicial object $X:\Delta^{\mathrm{op}}_{n}\to\Cc$ in a pointed category with finite limits $\Cc$ and the simplicial object $\mathrm{cosk}_{n}(X)$, since $\mathrm{tr}_{n}\mathrm{cosk}_{n}(X)=X$ and then $\mathrm{cosk}_{n}\mathrm{tr}_{n}(\mathrm{cosk}_{n}(X))=\mathrm{cosk}_{n}(X)$, by applying the previous result one obtains:

\begin{corollary}\label{cor:simpobj}
    Let $\Cc$ be a pointed category with finite limits and $X$ an object in $\mathrm{Simp}_{n}(\Cc)$. Then 
 \[   
    M(\mathrm{cosk}_{n}(X))_{i}=\mathbf{0}\ \text{if}\ i>n+1,\ M(\mathrm{cosk}_{n}(X))_{n+1}=\mathrm{Ker}(\partial_{n}:M(\mathrm{cosk}_{n}(X))_{n}\to M(\mathrm{cosk}_{n}(X))_{n-1}).
\]
\end{corollary}

\noindent\textbf{The category $\Hc$}. We denote by $\Hc$ the category of cocommutative color Hopf algebras, i.e., cocommutative Hopf monoids in the category of $G$-vector spaces Vec$_{G}$, where $G$ is an abelian group. Given $X=\bigoplus_{g\in G}{X_{g}}$ and $Y=\bigoplus_{g\in G}{Y_{g}}$ in Vec$_{G}$, the category Vec$_{G}$ is braided with braiding $c_{X,Y}:X\otimes Y\to Y\otimes X$ defined by $c_{X,Y}(x\otimes y)=\phi(g,h)y\otimes x$ for $x\in X_{g}$, $y\in Y_{h}$ and $g,h\in G$ on the components of the grading and extended by linearity, where $\phi:G\times G\to\Bbbk-\{0\}$ is a bicharacter on $G$, i.e., it satisfies
\[
\phi(gh,l)=\phi(g,l)\phi(h,l)\ \text{and}\ \phi(g,hl)=\phi(g,h)\phi(g,l)\ \text{for every}\ g,h,l\in G.
\]
Moreover, Vec$_{G}$ is symmetric if $\phi$ satisfies further $\phi(g,h)\phi(h,g)=1_{\Bbbk}$ for all $g,h\in G$. The objects of the category $\mathrm{Hopf}_{\mathrm{coc}}(\mathrm{Vec}_{G})$ are $G$-graded algebras $(H,m,u)$ which are also cocommutative $G$-graded coalgebras $(H,\Delta,\epsilon)$ with a compatibility between the two structures, equipped with an antipode $S$. The morphisms in $\Hc$ are simply algebra maps which are also coalgebra maps preserving gradings. We will use the same notations employ in \cite{As}. In particular, the compatibility condition reads as $\Delta(ab)=\phi(|a_{2}|,|b_{1}|)a_{1}b_{1}\otimes a_{2}b_{2}$ for all $a,b\in H$, while the cocommutativity of $H$ reads as $a_{1}\otimes a_{2}=\phi(|a_{1}|,|a_{2}|)a_{2}\otimes a_{1}$ for all $a\in H$. Recall also that the cocommutativity implies $\Delta(S(a))=S(a_{1})\otimes S(a_{2})$ for all $a\in H$, while $S(ab)=\phi(|a|,|b|)S(b)S(a)$ for all $a,b\in H$. We refer the reader to \cite{Ag} for many details about monoidal categories and to \cite{Sw} for basic results in Hopf algebra theory. \medskip

In \cite[Theorem 6.1]{As} we showed that the category $\Hc$ is semi-abelian if the abelian group $G$ is finitely generated and char$\Bbbk\not=2$ (not needed in case $G$ is finite of odd cardinality). Note that, in case $G=\{1\}$ is the trivial group, one recovers \cite[Theorem 2.10]{MG} for the category Hopf$_{\Bbbk,\mathrm{coc}}$ of cocommutative Hopf algebras.
From now on we suppose that $G$ is a finitely generated abelian group and char$\Bbbk\not=2$ and we study the categorical properties discussed above for $\Hc$.

\section{Huq commutators in $\Hc$}

In this section we give an explicit description of commutators in $\Hc$. 

\begin{remark}\label{p}
In $\Hc$ monomorphisms are exactly the injective morphisms as it is shown in \cite[Lemma 5.22]{As} and, since every injective map $f$ in $\Hc$ can be decomposed as an isomorphism followed by an inclusion, a subobject of an object $A$ in $\Hc$ is an inclusion $i:X\to A$ in $\Hc$, i.e., $X$ is a color Hopf subalgebra of $A$. Given $X$ and $Y$ color Hopf subalgebras of $A$,
recall that $X\times Y=X\otimes Y$ is the binary product of $X$ and $Y$ in $\Hc$ and $\langle{\mathrm{Id}_{X},0_{X,Y}}\rangle=(\mathrm{Id}_{X}\otimes u_{Y}\epsilon_{X})\circ\Delta_{X}:X\to X\otimes Y$, $x\mapsto x\otimes1_{Y}$, $\langle{0_{Y,X},\mathrm{Id}_{Y}}\rangle=(u_{X}\epsilon_{Y}\otimes\mathrm{Id}_{Y})\circ\Delta_{Y}:Y\to X\otimes Y$, $y\mapsto1_{X}\otimes y$. 
\end{remark}

Given $i:X\to A$ and $j:Y\to A$ two color Hopf subalgebras of $A$, these commute in the sense of Huq if there exists a morphism $p:X\otimes Y\to A$ in $\Hc$ such that $p(x\otimes1_{Y})=x$ and $p(1_{X}\otimes y)=y$ for every $x\in X$ and $y\in Y$. But then $p$ is uniquely determined, if it exists, since it has to be a morphism of algebras and so
\[
p(x\otimes y)=p((x\otimes1_{Y})(1_{X}\otimes y))=p(x\otimes1_{Y})p(1_{X}\otimes y)=xy\ \text{for all}\ x\in X,\ y\in Y.
\]
Note that the uniqueness of $p$, in case of existence, was already known since $\Hc$ is protomodular.
Hence we must have $p=m_{A}\circ(i\otimes j)$ which is a morphism of coalgebras since this is true for $m_{A}$ with $A$ color Hopf algebra. Thus we only need that $p$ is a morphism of algebras.

\begin{lemma}\label{l1}
Given two color Hopf subalgebras $i:X\to A$ and $j:Y\to A$ of $A$ in $\Hc$, the following conditions are equivalent:
\begin{enumerate}
\item[1)] there exists a unique morphism of color Hopf algebras $p:X\otimes Y\to A$ such that $i$ and $j$ commute in the sense of Huq; \medskip
\item[2)] $xy=\phi(|x|,|y|)yx$, for all $x\in X$ and $y\in Y$;\medskip
\item[3)] $\phi(|x_{2}|,|y_{1}|)x_{1}y_{1}S(x_{2})S(y_{2})=\epsilon(x)\epsilon(y)1_{A}$, for all $x\in X$ and $y\in Y$;\medskip
\item[4)] $\phi(|x_{2}|,|y|)x_{1}yS(x_{2})=\epsilon(x)y$, for all $x\in X$ and $y\in Y$.
\end{enumerate}
\end{lemma}

\begin{proof}
If 1) holds true, then $p$ is a morphism of algebras and so we obtain
\[
xy=p(\phi(|x|,|y|)\phi(|y|,|x|)x\otimes y)=\phi(|x|,|y|)p((1_{X}\otimes y)(x\otimes1_{Y}))=\phi(|x|,|y|)p(1_{X}\otimes y)p(x\otimes1_{Y})=\phi(|x|,|y|)yx
\]
for all $x\in X$ and $y\in Y$ and then 2) is satisfied, while if 2) holds true then 
\[
pm_{X\otimes Y}(x\otimes y\otimes x'\otimes y')=\phi(|y|,|x'|)xx'yy'\overset{2)}{=}xyx'y'=m_{A}(p\otimes p)(x\otimes y\otimes x'\otimes y')
\]
for every $x,x'\in X$ and $y,y'\in Y$ and then 1) is satisfied, 
since $p$ is a morphism of algebras. Thus 1) and 2) are equivalent conditions and now we show the equivalence between 2) and 3). If 2) is satisfied then we immediately obtain 3) as 
\[
\phi(|x_{2}|,|y_{1}|)x_{1}y_{1}S(x_{2})S(y_{2})\overset{2)}{=}x_{1}S(x_{2})y_{1}S(y_{2})=\epsilon(x)\epsilon(y)1_{A},
\]
while if 3) holds true then
\[
\begin{split}
xy&=x_{1}\epsilon(x_{2})y=\phi(|x_{2}|,|y|)x_{1}y\epsilon(x_{2})=\phi(|x_{2}\otimes x_{3}|,|y|)x_{1}y_{1}\epsilon(y_{2})S(x_{2})x_{3}\\&=\phi(|x_{2}|,|y_{1}|)\phi(|x_{3}|,|y|)x_{1}y_{1}S(x_{2})\epsilon(y_{2})x_{3}=\phi(|x_{2}|,|y_{1}|)\phi(|x_{3}|,|y|)x_{1}y_{1}S(x_{2})S(y_{2})y_{3}x_{3}\\&\overset{3)}{=}\phi(|x_{2}|,|y|)\epsilon(x_{1})\epsilon(y_{1})y_{2}x_{2}=\phi(|x_{2}|,|y|)\epsilon(x_{1})yx_{2}=\phi(|x_{1}|,|y|)\phi(|x_{2}|,|y|)y\epsilon(x_{1})x_{2}\\&=\phi(|x|,|y|)yx
\end{split}
\]
so 2) is satisfied. Finally, we show that 3) and 4) are equivalent. Clearly, if 4) holds true then
\[
\phi(|x_{2}|,|y_{1}|)x_{1}y_{1}S(x_{2})S(y_{2})\overset{4)}{=}\epsilon(x)y_{1}S(y_{2})=\epsilon(x)\epsilon(y)1_{A},
\]
i.e., 3) is satisfied, while, by assuming 3), we can compute
\[
\begin{split}
\phi(|x_{2}|,|y|)x_{1}yS(x_{2})&=\phi(|x_{2}|,|y_{1}\otimes y_{2}|)x_{1}y_{1}\epsilon(y_{2})S(x_{2})=\phi(|x_{2}|,|y_{1}|)x_{1}y_{1}S(x_{2})\epsilon(y_{2})\\&=\phi(|x_{2}|,|y_{1}|)x_{1}y_{1}S(x_{2})S(y_{2})y_{3}\overset{3)}{=}\epsilon(x)\epsilon(y_{1})y_{2}=\epsilon(x)y
\end{split}
\]
and then 4) holds true.
\end{proof}

Since the category $\Hc$ is semi-abelian, a normal subobject $i:X\to A$ of $A$ in $\Hc$ is an inclusion which is a kernel of a morphism in $\Hc$ and this is equivalent, by \cite[Corollary 5.21]{As}, to $X$ being a \textit{normal} color Hopf subalgebra of $A$, i.e., such that $\xi_{A}(a\otimes x):=\phi(|a_{2}|,|x|)a_{1}xS(a_{2})\in X$ for every $a\in A$ and $x\in X$. Recall from \cite[Lemma 5.4 1) and 2)]{As} that $\xi_{A}:A\otimes A\to A$ is a morphism of graded coalgebras and, given $f:A\to B$ in $\Hc$, then $\xi_{B}\circ(f\otimes f)=f\circ\xi_{A}$. First we can show the following result:

\begin{lemma}\label{AinAVecG}
   Given $A$ in $\Hc$, then 
   \begin{equation}\label{actionxi}
    \xi_{A}(a\otimes\xi_{A}(b\otimes c))=\xi_{A}(ab\otimes c)\ \text{and}\ \xi_{A}(1_{A}\otimes a)=a\ \text{for all}\ a,b,c\in A, 
\end{equation}
i.e., $(A,\xi_{A})$ is in $_{A}\mathrm{Vec}_{G}$.
\end{lemma}

\begin{proof}
Clearly $\xi_{A}$ preserves the gradings. Given $a,b,c\in A$, we can compute
\[
\begin{split}
    \xi_{A}(a\otimes\xi_{A}(b\otimes c))&=\phi(|b_{2}|,|c|)\xi_{A}(a\otimes b_{1}cS(b_{2}))=\phi(|b_{2}|,|c|)\phi(|a_{2}|,|b_{1}cS(b_{2})|)a_{1}b_{1}cS(b_{2})S(a_{2})\\&=\phi(|b_{2}|,|c|)\phi(|a_{2}|,|b_{1}c|)a_{1}b_{1}cS(a_{2}b_{2})=\phi(|a_{2}b_{2}|,|c|)\phi(|a_{2}|,|b_{1}|)a_{1}b_{1}cS(a_{2}b_{2})\\&=\phi(|(ab)_{2}|,|c|)(ab)_{1}cS((ab)_{2})=\xi_{A}(ab\otimes c)
\end{split}
\]
and $\xi_{A}(1_{A}\otimes a)=\phi(1_{G},|a|)1_{A}aS(1_{A})=a$ and so \eqref{actionxi} is satisfied.
\end{proof}

\begin{remark}
Observe that \eqref{actionxi} can also be deduced from \cite[Proposition 3.7.1]{Hec} seeing $A$ in $_{A}(\mathrm{Vec}_{G})_{A}$ with $A$-actions given by the multiplication $m_{A}$, in which case $\xi_{A}$ becomes $\mathrm{ad}_{A}$. Also note that 4) of Lemma \ref{l1} tells us that, when $i:X\to A$ and $j:Y\to A$ commute in $\Hc$, the restriction of the action $\xi_{A}$ to $X$ and $Y$ is trivial. From now on we set $\xi_{A}(a\otimes b):=a\triangleright b$, for all $a,b\in A$. In case there will be different $A$ and $B$ in $\Hc$ concurrently, the action $\triangleright$ will be clear from the context and we will omit to put down indexes.
\end{remark}

The Huq commutator of two normal color Hopf subalgebras $X$ and $Y$ of $A$ in $\Hc$, denoted by $[X,Y]_{\mathrm{Huq}}$, is then defined as the smallest normal color Hopf subalgebra of $A$ such that, considered the cokernel of the inclusion $[X,Y]_{\mathrm{Huq}}\to A$ in $\Hc$ which is given by the canonical projection $q:A\to A/A[X,Y]^{+}_{\mathrm{Huq}}$, the color Hopf subalgebras $q(X)$ and $q(Y)$ of $A/A[X,Y]^{+}_{\mathrm{Huq}}$, which are normal by \cite[Lemma 5.4, 2)]{As}, commute in the sense of $\mathrm{Huq}$. 
The latter means that there exists $p:q(X)\otimes q(Y)\to A/A[X,Y]^{+}_{\mathrm{Huq}}$, $\overline{x}\otimes\overline{y}\mapsto\overline{xy}$ in $\Hc$ which is equivalent, by Lemma \ref{l1}, to the condition $\overline{xy}=\phi(|x|,|y|)\overline{yx}$ for all $x\in X$ and $y\in Y$ and then to $xy-\phi(|x|,|y|)yx\in A[X,Y]^{+}_{\mathrm{Huq}}$ for all $x\in X$ and $y\in Y$. \medskip

Now we give an explicit description of the Huq commutator of two normal color Hopf subalgebras $X$ and $Y$ of $A$ in $\Hc$. By analogy with the notation used in \cite{MG}, we write $[X,Y]$ for the graded subalgebra of $A$ generated by all the elements of the form $[x,y]:=\phi(|x_{2}|,|y_{1}|)x_{1}y_{1}S(x_{2})S(y_{2})$ for any $x\in X$ and any $y\in Y$. Clearly we have $[x,y]=(x\triangleright y_{1})S(y_{2})$ for all $x\in X$ and $y\in Y$. Note also that
\begin{equation}\label{propscambio}
\begin{split}
[x,y]=\phi(|x_{2}|,|y_{1}|)x_{1}y_{1}S(x_{2})S(y_{2})=\phi(|x_{2}|,|y|)x_{1}(y\triangleright S(x_{2}))\ \text{for all}\ x\in X\ \text{and}\ y\in Y.
\end{split}
\end{equation}

The action $\triangleright$ satisfies some compatibility conditions with the antipode and the comultiplication of $A$, as shown in the following result.

\begin{lemma}
    Given $A$ in $\Hc$, the following properties are satisfied:
\begin{equation}\label{xiantipode}
    S(a\triangleright b)
    =a\triangleright S(b)\ \text{for all}\ a,b\in A,
\end{equation}
\begin{equation}\label{xiDelta}
    a\triangleright bc=\phi(|a_{2}|,|b|)(a_{1}\triangleright b)(a_{2}\triangleright c)\ \text{for all}\ a,b,c\in A.
\end{equation}
\end{lemma}

\begin{proof}
Given $a,b,c\in A$, we can compute
\[
\begin{split}
S(a\triangleright b)&=S(\phi(|a_{2}|,|b|)a_{1}bS(a_{2}))=\phi(|a_{2}|,|b|)\phi(|a_{1}b|,|a_{2}|)S(S(a_{2}))S(a_{1}b)=\phi(|a_{1}|,|a_{2}|)a_{2}S(a_{1}b)\\&=a_{1}S(a_{2}b)=\phi(|a_{2}|,|b|)a_{1}S(b)S(a_{2})=a\triangleright S(b),
\end{split}
\]
so that \eqref{xiantipode} holds true and
\[
\begin{split}
    a\triangleright bc&=\phi(|a_{2}|,|bc|)a_{1}bcS(a_{2})=\phi(|a_{3}|,|bc|)a_{1}\epsilon(a_{2})bcS(a_{3})=\phi(|a_{2}|,|b|)\phi(|a_{3}|,|bc|)a_{1}b\epsilon(a_{2})cS(a_{3})\\&=\phi(|a_{2}\otimes a_{3}|,|b|)\phi(|a_{4}|,|bc|)a_{1}bS(a_{2})a_{3}cS(a_{4})=\phi(|a_{2}|,|b|)\phi(|a_{3}|,|bc|)(a_{1}\triangleright b)a_{2}cS(a_{3})\\&=\phi(|a_{2}\otimes a_{3}|,|b|)\phi(|a_{3}|,|c|)(a_{1}\triangleright b)a_{2}cS(a_{3})=\phi(|a_{2}|,|b|)(a_{1}\triangleright b)(a_{2}\triangleright c)
\end{split}
\]
and then also \eqref{xiDelta} is satisfied. 
\end{proof}

Recall that we already know that $\xi_{A}$ is a morphism of graded coalgebras for all $A$ in $\Hc$, i.e., we have
\begin{equation}\label{xiofcoalg}
(a\triangleright b)_{1}\otimes(a\triangleright b)_{2}=\phi(|a_{2}|,|b_{1}|)(a_{1}\triangleright b_{1})\otimes(a_{2}\triangleright b_{2})\ \text{and}\ \epsilon(a\triangleright b)=\epsilon(a)\epsilon(b)
\end{equation}
for all $a,b\in A$. Moreover, we also know that, given $f:A\to B$ in $\Hc$, we have $f(a\triangleright b)=f(a)\triangleright f(b)$ for all $a,b\in A$.

\begin{proposition}\label{[X,Y]normal}
The graded algebra $[X,Y]$ is a normal color Hopf subalgebra of $A$.
\end{proposition}

\begin{proof}
By definition, $[X,Y]$ is a graded subalgebra of $A$. Thus, if we show that it is a graded subcoalgebra of $A$ and that it is closed under the antipode of $A$, we obtain that it is a color Hopf subalgebra of $A$. 
From the fact that $\xi_{A}$ is a morphism of graded coalgebras, 
we obtain that $[X,Y]$ is a subcoalgebra of $A$. 
Indeed, we can compute
\[
\begin{split}
    \Delta_{A}([x,y])&=\Delta_{A}((x\triangleright y_{1})S(y_{2}))=\phi(|(x\triangleright y_{1})_{2}|,|S(y_{2})_{1}|)(x\triangleright y_{1})_{1}S(y_{2})_{1}\otimes(x\triangleright y_{1})_{2}S(y_{2})_{2}\\&\overset{\eqref{xiofcoalg}}{=}\phi(|x_{2}|,|y_{1}|)\phi(|x_{2}\triangleright y_{2}|,|S(y_{3})|)(x_{1}\triangleright y_{1})S(y_{3})\otimes(x_{2}\triangleright y_{2})S(y_{4})\\&=\phi(|x_{2}|,|y_{1}|)\phi(|x_{2}|,|y_{2}|)(x_{1}\triangleright y_{1})S(y_{2})\otimes(x_{2}\triangleright y_{3})S(y_{4})\\&=\phi(|x_{2}|,|y_{1}|)[x_{1},y_{1}]\otimes[x_{2},y_{2}]\in[X,Y]\otimes[X,Y].
\end{split}
\]
Moreover, we have 
\[
\begin{split}
     S_{A}([x,y])&=S_{A}((x\triangleright y_{1})S_{A}(y_{2}))=\phi(|x\otimes y_{1}|,|y_{2}|)y_{2}S_{A}(x\triangleright y_{1})=\phi(|x|,|y_{1}|)y_{1}S_{A}(x\triangleright y_{2})\\&\overset{\eqref{xiantipode}}{=}\phi(|x|,|y_{1}|)y_{1}(x\triangleright S(y_{2}))
     \overset{\eqref{propscambio}}{=}\phi(|x|,|y|)[y,x]\in[X,Y],
\end{split}
\]
hence $[X,Y]$ is a color Hopf subalgebra of $A$. Finally, we show that $[X,Y]$ is normal:
\[
\begin{split}
    a\triangleright[x,y]&=a\triangleright((x\triangleright y_{1})S(y_{2}))\overset{\eqref{xiDelta}}{=}\phi(|a_{2}|,|x\triangleright y_{1}|)(a_{1}\triangleright(x\triangleright y_{1}))(a_{2}\triangleright S(y_{2}))\\&\overset{\eqref{actionxi}}{=}\phi(|a_{2}|,|x\otimes y_{1}|)(a_{1}x\triangleright y_{1})(a_{2}\triangleright S(y_{2}))=\phi(|a_{2}\otimes a_{3}|,|x\otimes y_{1}|)(a_{1}x\triangleright y_{1})(\epsilon(a_{2})a_{3}\triangleright S(y_{2}))\\&=\phi(|a_{2}|,|x|)\phi(|a_{3}|,|x\otimes y_{1}|)(a_{1}x\epsilon(a_{2})\triangleright y_{1})(a_{3}\triangleright S(y_{2}))\\&=\phi(|a_{2}\otimes a_{3}|,|x|)\phi(|a_{4}|,|x\otimes y_{1}|)(a_{1}xS(a_{2})a_{3}\triangleright y_{1})(a_{4}\triangleright S(y_{2}))\\&\overset{\eqref{actionxi}}{=}\phi(|a_{2}\otimes a_{3}|,|x|)\phi(|a_{4}|,|x\otimes y_{1}|)((a_{1}xS(a_{2}))\triangleright(a_{3}\triangleright y_{1}))(a_{4}\triangleright S(y_{2}))\\&\overset{\eqref{xiantipode}}{=}\phi(|a_{3}|,|x|)\phi(|a_{4}|,|x\otimes y_{1}|)((\phi(|a_{2}|,|x|)a_{1}xS(a_{2}))\triangleright(a_{3}\triangleright y_{1}))S(a_{4}\triangleright y_{2})\\&=\phi(|a_{2}|,|x|)\phi(|a_{3}|,|x\otimes y_{1}|)((a_{1}\triangleright x)\triangleright(a_{2}\triangleright y_{1}))S(a_{3}\triangleright y_{2})\\&=\phi(|a_{2}\otimes a_{3}|,|x|)\phi(|a_{3}|,|y_{1}|)((a_{1}\triangleright x)\triangleright(a_{2}\triangleright y_{1}))S(a_{3}\triangleright y_{2})\\&=\phi(|a_{2}|,|x|)((a_{1}\triangleright x)\triangleright(a_{2}\triangleright y)_{1})S((a_{2}\triangleright y)_{2})\\&=\phi(|a_{2}|,|x|)[(a_{1}\triangleright x),(a_{2}\triangleright y)]\in[X,Y]
\end{split}
\]
and then the thesis follows.
\end{proof}

\begin{proposition}\label{prop:commutator}
    Given $X$ and $Y$ normal color Hopf subalgebras of $A$, then $[X,Y]=[X,Y]_{\mathrm{Huq}}$.
\end{proposition}

\begin{proof}
Recall that $[X,Y]_{\mathrm{Huq}}$ is defined as the smallest normal color Hopf subalgebra of $A$ such that 
\begin{equation}\label{eqHuq}
xy-\phi(|x|,|y|)yx\in A[X,Y]_{\mathrm{Huq}}^{+}=[X,Y]_{\mathrm{Huq}}^{+}A 
\end{equation}
for all $x\in X$ and $y\in Y$. By Proposition \ref{[X,Y]normal} we already know that $[X,Y]$ is a normal color Hopf subalgebra of $A$, hence we only have to show that it is the smallest which satisfies \eqref{eqHuq}. Observe that
\[
\begin{split}
xy&=x_{1}\epsilon(x_{2})y=\phi(|x_{2}|,|y|)x_{1}y\epsilon(x_{2})=\phi(|x_{2}\otimes x_{3}|,|y|)x_{1}y_{1}\epsilon(y_{2})S(x_{2})x_{3}\\&=\phi(|x_{2}|,|y_{1}|)\phi(|x_{3}|,|y|)x_{1}y_{1}S(x_{2})\epsilon(y_{2})x_{3}=\phi(|x_{2}|,|y_{1}|)\phi(|x_{3}|,|y|)x_{1}y_{1}S(x_{2})S(y_{2})y_{3}x_{3}\\&=\phi(|x_{2}|,|y|)[x_{1},y_{1}]y_{2}x_{2}
\end{split}
\]
and 
\[
\begin{split}
\phi(|x|,|y|)yx&=\phi(|x_{2}|,|y|)\epsilon(x_{1})\epsilon(y_{1})y_{2}x_{2}=\phi(|x_{3}|,|y|)\epsilon(x_{1})\epsilon(x_{2})\epsilon(y_{1})\epsilon(y_{2})y_{3}x_{3}\\&=\phi(|x_{3}|,|y|)\phi(|x_{2}|,|y_{1}|)\epsilon(x_{1})\epsilon(y_{1})\epsilon(S(x_{2}))\epsilon(S(y_{2}))y_{3}x_{3}\\&=\phi(|x_{3}|,|y|)\epsilon(\phi(|x_{2}|,|y_{1}|)x_{1}y_{1}S(x_{2})S(y_{2}))y_{3}x_{3}\\&=\phi(|x_{2}|,|y|)\epsilon([x_{1},y_{1}])y_{2}x_{2},
\end{split}
\]
so that we obtain
\[
xy-\phi(|x|,|y|)yx=\big([x_{1},y_{1}]-\epsilon([x_{1},y_{1}])\big)\phi(|x_{2}|,|y|)y_{2}x_{2}\in[X,Y]^{+}A,
\]
hence $[X,Y]$ satisfies \eqref{eqHuq}. 
Finally, we prove that $[X,Y]$ is the smallest normal color Hopf subalgebra which satisfies \eqref{eqHuq} showing that $[X,Y]\subseteq\mathrm{Hker}(f)$ for every morphism $f:A\to B$ in $\Hc$ such that $f(X)$ and $f(Y)$ commute in $B$. Hence, given $[x,y]=(x\triangleright y_{1})S(y_{2})$ in $[X,Y]$ for $x\in X$ and $y\in Y$, we can compute
\[
\begin{split}
    [x,y]_{1}\otimes f([x,y]_{2})&=((x\triangleright y_{1})S(y_{2}))_{1}\otimes f(((x\triangleright y_{1})S(y_{2}))_{2})\\&=\phi(|(x\triangleright y_{1})_{2}|,|S(y_{2})_{1}|)(x\triangleright y_{1})_{1}S(y_{2})_{1}\otimes f((x\triangleright y_{1})_{2}S(y_{2})_{2})\\&\overset{\eqref{xiofcoalg}}{=}\phi(|x_{2}\otimes y_{2}|,|y_{3}|)\phi(|x_{2}|,|y_{1}|)(x_{1}\triangleright y_{1})S(y_{3})\otimes f((x_{2}\triangleright y_{2})S(y_{4}))\\&=\phi(|x_{2}|,|y_{1}|)\phi(|x_{2}|,|y_{2}|)(x_{1}\triangleright y_{1})S(y_{2})\otimes f(x_{2}\triangleright y_{3})f(S(y_{4}))\\&=\phi(|x_{2}|,|y_{1}|)\phi(|x_{2}|,|y_{2}|)(x_{1}\triangleright y_{1})S(y_{2})\otimes(f(x_{2})\triangleright f(y_{3}))f(S(y_{4}))\\&
    \overset{(*)}{=}\phi(|x_{2}|,|y_{1}|)\phi(|x_{2}|,|y_{2}|)(x_{1}\triangleright y_{1})S(y_{2})\otimes \epsilon(f(x_{2}))f(y_{3})f(S(y_{4}))\\&=\phi(|x_{2}|,|y_{1}|)\phi(|x_{2}|,|y_{2}|)(x_{1}\triangleright y_{1})S(y_{2})\otimes\epsilon(x_{2})f(y_{3}S(y_{4}))\\&=\phi(|x_{2}|,|y_{1}|)\phi(|x_{2}|,|y_{2}|)(x_{1}\triangleright y_{1})S(y_{2})\otimes\epsilon(x_{2})\epsilon(y_{3})1_{B}\\&=(x\triangleright y_{1})S(y_{2})\otimes\epsilon(y_{3})1_{B}=[x,y]\otimes1_{B}
\end{split}
\]
where $(*)$ holds true since $f(X)$ and $f(Y)$ commute in $B$, using 4) of Lemma \ref{l1}.
Hence $[X,Y]$ is the Huq commutator of $X$ and $Y$ in $\Hc$.
\end{proof}

The previous result generalizes \cite[Proposition 4.3]{MG} given for the category of cocommutative Hopf algebras. \medskip

\section{The Hall’s criterion and the Zassenhaus Lemma in $\Hc$}

A color Hopf algebra $A$ is called \textit{nilpotent} if $\gamma^{n}_{A}(A)=0$ for a certain $n\in\mathbb{N}$ where, denoted by $\mathrm{Norm}(A)$ the set of all normal color Hopf subalgebras of $A$, the map $\gamma_{A}:\mathrm{Norm}(A)\to\mathrm{Norm}(A)$ sends a normal color Hopf subalgebra $X$ of $A$ to $[A,X]_{\mathrm{Huq}}$, which is $[A,X]$ by Proposition \ref{prop:commutator}. The least such $n$ is the nilpotency class of $A$. Recall that, by \cite[Lemma 5.22]{As}, regular epimorphisms in $\Hc$ are exactly the surjective maps. Moreover, in \cite[Proposition 6.3]{As} we showed that $\Hc$ is locally algebraically cartesian closed, then algebraically coherent by \cite[Theorem 4.5]{CGV}. Hence we have the Hall's criterion for nilpotence in $\Hc$, given by the following result in view of Theorem \ref{nilpotence}:

\begin{proposition}
Let $p:E\to B$ be a surjective map in $\Hc$, where $B$ is nilpotent in $\Hc$. If $\mathrm{Hker}(p)\subseteq[N,N]$ for a nilpotent normal color Hopf subalgebra $N$ of $E$, then $E$ is nilpotent. Furthermore, if $N$ is of nilpotency class $c$ and $B$ is of nilpotency class $d$, then $E$ is of nilpotency class at most $\frac{c(c+1)}{2}(d-1)+c$.
\end{proposition}

Moreover, we can also describe the Zassenhaus Lemma in $\Hc$.

\begin{lemma}\label{KM=MK}
    Let $M$ and $K$ be two color Hopf subalgebras of a cocommutative color Hopf algebra $A$, with $K$ normal in $A$. Then, the graded vector spaces $KM$ and $MK$, generated by linear combinations of elements of the form $km$ and $mk$ with $k\in K$ and $m\in M$ respectively, are the same object in $\mathrm{Vec}_{G}$.
\end{lemma}

\begin{proof}
Since $MK$ and $KM$ are graded subspaces of $A$, in order to conclude it is sufficient to show that they are the same vector space. Given $k\in K$ and $m\in M$, we have
\[
\begin{split}
k&m=k\epsilon(m_{1})m_{2}=\phi(|k|,|m_{1}|)\epsilon(m_{1})km_{2}=\phi(|k|,|m_{1}\otimes m_{2}|)m_{1}S(m_{2})kS(S(m_{3}))\\&=\phi(|k|,|m|)\phi(|m_{3}|,|k|)m_{1}S(m_{2})kS(S(m_{3}))=
\phi(|k|,|m|)\phi(|S(m_{2})_{2}|,|k|)m_{1}S(m_{2})_{1}kS(S(m_{2})_{2})\in MK,
\end{split}
\]
since $K$ is normal, hence $KM\subseteq MK$. Moreover, we also have
\[
\begin{split}
mk&=m_{1}\epsilon(m_{2})k=\phi(|m_{2}|,|k|)m_{1}k\epsilon(m_{2})=\phi(|m_{2}|,|k|)\phi(|m_{3}|,|k|)m_{1}kS(m_{2})m_{3}
\in KM,
\end{split}
\]
using again that $K$ is normal 
and so the thesis follows.
\end{proof}

\begin{lemma}
    Let $M$ and $K$ be two color Hopf subalgebras of a cocommutative color Hopf algebra $A$, with $K$ normal in $A$. Then the graded vector space $KM$ is a color Hopf subalgebra of $A$.
\end{lemma}

\begin{proof}
Clearly $KM$ contains $1_{A}$, thus we only have to show that it is closed under $m_{A}$, $\Delta_{A}$ and $S_{A}$. Given $k,k'\in K$ and $m,m'\in M$, since $MK=KM$ in Vec$_{G}$ by Lemma \ref{KM=MK}, we obtain 
\[
kmk'm'\in KMKM=KKMM=KM,\ \ \Delta(km)=\phi(|k_{2}|,|m_{1}|)k_{1}m_{1}\otimes k_{2}m_{2}\in KM\otimes KM
\]
and $S(km)=\phi(|k|,|m|)S(m)S(k)\in MK=KM$. Hence $KM$ is a color Hopf subalgebra of $A$, then automatically cocommutative.
\end{proof}

\begin{proposition}
    Let $M$ and $K$ be two color Hopf subalgebras of a cocommutative color Hopf algebra $A$ with $K$ normal in $A$. Then $KM$ is the supremum of $K$ and $M$ as color Hopf subalgebras of $A$, i.e., $K\vee_{A}M=KM$.
\end{proposition}

\begin{proof}
    If there exists a color Hopf subalgebra $L$ of $A$ containing $K$ and $M$, then it clearly contains also $KM$, because a color Hopf subalgebra is closed under products and sums.
\end{proof}

Hence, by 
applying the categorical results in \cite{OlSt}, we obtain the analogous of \cite[Propositions 4.4 and 4.5]{OlSt} and \cite[Theorem 4.6]{OlSt}. Let us make explicit the Zassenhaus Lemma in the setting of cocommutative color Hopf algebras. Recall that, given a cocommutative color Hopf algebra $A$ and a normal color Hopf subalgebra $B$ of $A$, the cokernel object of the inclusion $i:B\to A$ is given by $A/AB^{+}$. Hence, in view of Theorem \ref{Zassenhaus}, we obtain the following:

\begin{proposition}
Let $U$ and $V$ be two color Hopf subalgebras of a cocommutative color Hopf algebra $A$, $K$ a normal color Hopf subalgebra of $U$ and $L$ a normal color Hopf subalgebra of $V$. Then the following 
\[
\frac{K(U\cap V)}{K(U\cap V)(K(L\cap U))^{+}}\cong\frac{U\cap V}{(U\cap V)((K\cap V)(L\cap U))^{+}}\cong\frac{L(U\cap V)}{L(U\cap V)(L(K\cap V))^{+}}
\]
are isomorphisms in $\Hc$.
\end{proposition}

\section{Color Hopf crossed modules as internal crossed modules}

In this section we give a characterization of the category $\mathrm{XMod}(\Hc)$ of internal crossed modules in the semi-abelian category $\Hc$ by providing an equivalence with another category, apparently different. In order to do this first we have to understand better the category $\mathrm{Pt}(\Hc)$. Recall also that, by what we said in the preliminary section, the category XMod$(\Hc)$ is equivalent to Grpd$(\Hc)$ which is isomorphic to RMG$(\Hc)$.

\begin{remark}
Recall that, given $A$ in Bimon(Vec$_{G})$, the category $_{A}\mathrm{Vec}_{G}$ is monoidal. The unit object is $\Bbbk$ with left $A$-action defined by $a\cdot k:=\epsilon(a)k$ for all $a\in A$ and $k\in\Bbbk$ and, given $V$ and $W$ in $_{A}\mathrm{Vec}_{G}$, the tensor product $V\otimes W$ has left $A$-action given by $a\cdot(v\otimes w):=\phi(|a_{2}|,|v|)(a_{1}\cdot v)\otimes(a_{2}\cdot w)$, for all $a\in A$, $v\in V$ and $w\in W$. Furthermore, if we consider $A$ in $\Hc$ then the category $_{A}\mathrm{Vec}_{G}$ is symmetric monoidal since Vec$_{G}$ is symmetric (see e.g. \cite[Proposition 6.40]{Ag}). Thus, we can consider the category Hopf($_{A}\mathrm{Vec}_{G}$), whose objects are color Hopf algebras $(H,m,u,\Delta,\epsilon,S)$ which are also graded left $A$-modules 
and such that $m,u,\Delta,\epsilon,S$ are morphisms of left $A$-modules. Hence, denoting by $\cdot:A\otimes H\to H$ the left $A$-action, the following relations are satisfied
\[
a\cdot1_{H}=\epsilon(a)1_{H},\ \ \epsilon(a\cdot h)=\epsilon(a)\epsilon(h), \ \ \Delta(a\cdot h)=\phi(|a_{2}|,|h_{1}|)(a_{1}\cdot h_{1})\otimes(a_{2}\cdot h_{2}), \ \ a\cdot(hh')=\phi(|a_{2}|,|h|)(a_{1}\cdot h)(a_{2}\cdot h')
\]
for every $a\in A$ and $h,h'\in H$. Note that we are considering the trivial braiding on the category $_{A}\mathrm{Vec}_{G}$, given by the braiding of $\mathrm{Vec}_{G}$, since $A$ is cocommutative. We call an object in Hopf($_{A}\mathrm{Vec}_{G}$) an \textit{$A$-module color Hopf algebra}. Note that we can consider $A$ as an $A$-module color Hopf algebra with left $A$-action given by $\triangleright$. In fact, $A$ is in $_{A}\mathrm{Vec}_{G}$ by Lemma \ref{AinAVecG}, clearly $a\triangleright1_{A}=\epsilon(a)1_{A}$ for all $a\in A$ and $\triangleright$ satisfies the second and third equations as said in \eqref{xiofcoalg}. Moreover, also the last equation holds true since $\triangleright$ satisfies \eqref{xiDelta}.  
\end{remark}

The fact that $\Hc$ is a semi-abelian category directly implies an equivalence between internal actions and split extensions of cocommutative color Hopf algebras \cite{BoJa}. As in the case of cocommutative Hopf algebras one can define the category of (cocommutative) module color Hopf algebras.

\begin{definition}
The category of (cocommutative) module color Hopf algebras is defined as follows. Objects are (cocommutative) $A$-module color Hopf algebras, with $A$ an arbitrary cocommutative color Hopf algebra and, given a (cocommutative) $A$-module color Hopf algebra $H$ and a (cocommutative) $B$-module color Hopf algebra $K$ with $A,B$ in $\Hc$, a morphism between them is given by a pair of morphisms of color Hopf algebras $\alpha:A\to B$ and $\beta:H\to K$ such that $\beta(a\cdot h)=\alpha(a)\cdot\beta(h)$, for $a\in A$ and $h\in H$. We denote the category of cocommutative module color Hopf algebras by Act($\Hc$). 
\end{definition}

The category Act($\Hc$) will turn out to be equivalent to the category Pt($\Hc$), so that it will coincide with the category of internal actions in $\Hc$, 
justifying the notation adopted.

Moreover, one can consider semi-direct products (also called smash products) starting from a general braided monoidal category and then, in case of Vec$_{G}$, one recovers the following definition:

\begin{definition}
Given $A$ in $\Hc$ and a (cocommutative) $A$-module color Hopf algebra $H$, the \textit{semi-direct product} of $H$ and $A$, denoted by $H\rtimes A$, is the (cocommutative) color Hopf algebra which is $H\otimes A$ as a $G$-graded vector space, an algebra with unit $1_{H\rtimes A}:=1_{H}\otimes1_{A}$ and multiplication defined by
\[
(h\otimes a)(h'\otimes a'):=\phi(|a_{2}|,|h'|)h(a_{1}\cdot h')\otimes a_{2}a'
\]
for all $h,h'\in H$ and $a,a'\in A$ and a coalgebra with counit and comultiplication defined by 
\[
\epsilon(h\otimes a):=\epsilon(h)\epsilon(a)\ \mathrm{and}\ \Delta(h\otimes a):=\phi(|h_{2}|,|a_{1}|)h_{1}\otimes a_{1}\otimes h_{2}\otimes a_{2}. 
\]
The antipode is given by $S(h\otimes a):=\phi(|h|,|a_{1}|)(S(a_{1})\cdot S(h))\otimes S(a_{2})$.
\end{definition}

The following result is a special case of \cite[Theorem 3.10.4]{Hec}, see also \cite{Bes} and \cite[Theorem 5.1.5]{BesDra} for the original sources. Note also that it is a straightforward generalization of the result proved in \cite{Mol} for the category $\mathrm{Hopf}_{\Bbbk,\mathrm{coc}}$.

\begin{proposition}\label{prop:fiso}
Let  $H$ and $A$ be color Hopf algebras with $A$ cocommutative and
\begin{tikzcd}
	H & A
	\arrow[shift left=1, from=1-1, to=1-2, "p"]
	\arrow[shift left=1, from=1-2, to=1-1,"i"]
\end{tikzcd}
a splitting of morphisms of color Hopf algebras, i.e., $p\circ i=\mathrm{Id}_{A}$. Then, $\mathrm{Hker}(p)$ is an $A$-module color Hopf algebra through
\[
\cdot:A\otimes\mathrm{Hker}(p)\to\mathrm{Hker}(p),\ a\otimes k\mapsto\phi(|a_{2}|,|k|)i(a_{1})ki(S(a_{2})).
\]
Furthermore, there exists an isomorphism of color Hopf algebras
\[
f:\mathrm{Hker}(p)\rtimes A\to H,\ k\otimes a\mapsto ki(a)
\]
with inverse $g:H\to\mathrm{Hker}(p)\rtimes A$, $h\mapsto h_{1}i(p(S(h_{2})))\otimes p(h_{3})$.
\end{proposition}

Note that $a\cdot k=i(a)\triangleright k$ for all $a\in A$ and $k\in\mathrm{Hker}(p)$.

\begin{invisible}
\begin{proof}
First we have to show that $\cdot$ is well-defined and so that, given $a\in A$ and $k\in\mathrm{Hker}(p)$, then $a\cdot k\in\mathrm{Hker}(p)$. So, we compute
\[
\begin{split}
    \Delta(a\cdot k)&=\Delta(\phi(|a_{2}|,|k|)i(a_{1})ki(S(a_{2})))\\&=\phi(|a_{4}|,|k|)\phi(|a_{3}|,|k|)\phi(|a_{2}|,|k_{1}|)(i(a_{1})k_{1}\otimes i(a_{2})k_{2})(i(S(a_{3}))\otimes i(S(a_{4})))\\&=\phi(|a_{4}|,|k|)\phi(|a_{3}|,|k|)\phi(|a_{2}|,|k_{1}|)\phi(|a_{2}k_{2}|,|a_{3}|)i(a_{1})k_{1}i(S(a_{3}))\otimes i(a_{2})k_{2}i(S(a_{4}))\\&=
\phi(|a_{4}|,|k|)\phi(|a_{3}|,|k_{1}|)\phi(|a_{2}|,|k_{1}|)i(a_{1})k_{1}i(S(a_{2}))\otimes i(a_{3})k_{2}i(S(a_{4}))
\end{split}
\]
and then
\[
\begin{split}  
    (\mathrm{Id}\otimes p)\Delta(a\cdot k)&=\phi(|a_{4}|,|k|)\phi(|a_{3}|,|k_{1}|)\phi(|a_{2}|,|k_{1}|)i(a_{1})k_{1}i(S(a_{2}))\otimes p(i(a_{3}))p(k_{2})p(i(S(a_{4})))\\&=\phi(|a_{4}|,|k|)\phi(|a_{3}|,|k_{1}|)\phi(|a_{2}|,|k_{1}|)i(a_{1})k_{1}i(S(a_{2}))\otimes a_{3}\epsilon(k_{2})S(a_{4})\\&=\phi(|a_{4}|,|k|)\phi(|a_{3}|,|k|)\phi(|a_{2}|,|k|)i(a_{1})ki(S(a_{2}))\otimes a_{3}S(a_{4})\\&=\phi(|a_{3}|,|k|)\phi(|a_{2}|,|k|)i(a_{1})ki(S(a_{2}))\otimes\epsilon(a_{3})1_{A}\\&=\phi(|a_{2}|,|k|)i(a_{1})ki(S(a_{2}))\otimes1_{A}=(a\cdot k)\otimes1_{A}
\end{split}
\]
where the second equality holds true since $k\in\mathrm{Hker}(p)$ and $p\circ i=\mathrm{Id}_{A}$. Hence $a\cdot k\in\mathrm{Hker}(p)$ and $\cdot$ is well-defined.
Furthermore, $\cdot$ is an $A$-action, indeed $1_{A}\cdot k=\phi(|1_{A}|,|k|)i(1_{A})ki(S(1_{A}))=k$ and 
\[
\begin{split}
a'\cdot(a\cdot k)&=a'\cdot(\phi(|a_{2}|,|k|)i(a_{1})ki(S(a_{2})))\\&=\phi(|a_{2}|,|k|)\phi(|a'_{2}|,|a_{1}ka_{2}|)i(a'_{1})i(a_{1})ki(S(a_{2}))i(S(a'_{2}))\\&=\phi(|a_{2}|,|k|)\phi(|a'_{2}|,|a_{1}ka_{2}|)i(a'_{1}a_{1})ki(S(a_{2})S(a'_{2}))\\&=\phi(|a_{2}|,|k|)\phi(|a'_{2}|,|a_{1}k|)i(a'_{1}a_{1})ki(S(a'_{2}a_{2}))\\&=\phi(|a'_{2}a_{2}|,|k|)\phi(|a'_{2}|,|a_{1}|)i(a'_{1}a_{1})ki(S(a'_{2}a_{2}))\\&=\phi(|(a'a)_{2}|,|k|)i((a'a)_{1})ki(S((a'a)_{2}))=(a'a)\cdot k.
\end{split}
\]
Moreover, we have 
\[
a\cdot1_{H}=\phi(|a_{2}|,|1_{H}|)i(a_{1})1_{H}i(S(a_{2}))=i(a_{1})i(S(a_{2}))=i(a_{1}S(a_{2}))=\epsilon(a)1_{H}
\]
and
\[
\begin{split}
\epsilon(a\cdot k)&=\epsilon(\phi(|a_{2}|,|k|)i(a_{1})ki(S(a_{2})))=\phi(|a_{2}|,|k|)\epsilon(i(a_{1}))\epsilon(k)\epsilon(i(S(a_{2})))\\&=\phi(|a_{2}|,|k|)\epsilon(a_{1})\epsilon(k)\epsilon(a_{2})=\epsilon(a_{1})\epsilon(a_{2})\epsilon(k)=\epsilon(a)\epsilon(k).
\end{split}
\]
Furthermore
\[
\begin{split}
    \phi(|a_{2}|,|k_{1}|)(a_{1}\cdot k_{1})\otimes(a_{2}\cdot k_{2})&=\phi(|a_{4}|,|k_{1}|)\phi(|a_{3}|,|k_{1}|)\phi(|a_{2}|,|k_{1}|)\phi(|a_{4}|,|k_{2}|)i(a_{1})k_{1}i(S(a_{2}))\otimes i(a_{3})k_{2}i(S(a_{4}))\\&=\phi(|a_{4}|,|k|)\phi(|a_{3}|,|k_{1}|)\phi(|a_{2}|,|k_{1}|)i(a_{1})k_{1}i(S(a_{2}))\otimes i(a_{3})k_{2}i(S(a_{4}))\\&=\Delta(a\cdot k)
\end{split}
\]
and, similarly, 
\[
\begin{split}
    \phi(|a_{2}|,|k|)(a_{1}\cdot k)(a_{2}\cdot k')&=\phi(|a_{4}|,|k|)\phi(|a_{3}|,|k|)\phi(|a_{2}|,|k|)i(a_{1})ki(S(a_{2}))i(a_{3})k'i(S(a_{4}))\\&=\phi(|a_{3}|,|k|)\phi(|a_{2}|,|k|)i(a_{1})k\epsilon(a_{2})k'i(S(a_{3}))\\&=\phi(|a_{2}|,|k|)\phi(|a_{2}|,|k'|)i(a_{1})kk'i(S(a_{2}))=\phi(|a_{2}|,|kk'|)i(a_{1})kk'i(S(a_{2}))\\&=a\cdot(kk').
\end{split}
\]
Thus, $\mathrm{Hker}(p)$ is a cocommutative $A$-module color Hopf algebra and we can consider $\mathrm{Hker}(p)\rtimes A$. Define $f:\mathrm{Hker}(p)\rtimes A\to H,\ k\otimes a\mapsto ki(a)$. This is a morphism of coalgebras, indeed
\[
\begin{split}
(f\otimes f)\Delta(k\otimes a)&=(f\otimes f)(\phi(|k_{2}|,|a_{1}|)k_{1}\otimes a_{1}\otimes k_{2}\otimes a_{2})=\phi(|k_{2}|,|a_{1}|)k_{1}i(a_{1})\otimes k_{2}i(a_{2})=\Delta(ki(a))=\Delta f(a)
\end{split}
\]
and $\epsilon f(k\otimes a)=\epsilon(ki(a))=\epsilon(k)\epsilon(a)$. Moreover, $f$ is a morphism of algebras since
\[
\begin{split}
f((k\otimes a)(k'\otimes a'))&=f(\phi(|a_{2}|,|k'|)k(a_{1}\cdot k')\otimes a_{2}a')=\phi(|a_{2}|,|k'|)k(a_{1}\cdot k')i(a_{2}a')\\&=\phi(|a_{3}|,|k'|)\phi(|a_{2}|,|k'|)ki(a_{1})k'i(S(a_{2}))i(a_{3}a')\\&=\phi(|a_{3}|,|k'|)\phi(|a_{2}|,|k'|)ki(a_{1})k'i(S(a_{2})a_{3}a')\\&=\phi(|a_{2}|,|k'|)ki(a_{1})k'i(\epsilon(a_{2})a')=ki(a_{1}\epsilon(a_{2}))k'i(a')\\&=f(k\otimes a)f(k'\otimes a').
\end{split}
\]
and $f(1_{H}\otimes1_{A})=1_{H}$, thus $f$ is a morphism of color Hopf algebras. In addition, we can define $g:H\to\mathrm{Hker}(p)\rtimes A$, $h\mapsto h_{1}i(p(S(h_{2})))\otimes p(h_{3})$, which is well-defined since
\[
\begin{split}
    (\mathrm{Id}\otimes p)\Delta(h_{1}i(p(S(h_{2}))))&=(\mathrm{Id}\otimes p)(\phi(|h_{2}|,|h_{3}|)h_{1}i(p(S(h_{3})))\otimes h_{2}i(p(S(h_{4}))))\\&=h_{1}i(p(S(h_{2})))\otimes p(h_{3})p(i(p(S(h_{4}))))\\&=h_{1}i(p(S(h_{2})))\otimes p(h_{3}S(h_{4}))\\&=h_{1}i(p(S(h_{2})))\otimes1_{A}
\end{split}
\]
and it is the inverse of $f$, indeed 
\[
fg(h)=f(h_{1}i(p(S(h_{2})))\otimes p(h_{3}))=h_{1}i(p(S(h_{2})))i(p(h_{3}))=h_{1}i(p(S(h_{2})h_{3}))=h_{1}\epsilon(h_{2})i(p(1_{H}))=h
\]
and
\[
\begin{split}
gf(k\otimes a)&=g(ki(a))=\phi(|k_{3}|,|a_{1}|)\phi(|k_{3}|,|a_{2}|)\phi(|k_{2}|,|a_{1}|)k_{1}i(a_{1})i(p(S(k_{2}i(a_{2}))))\otimes p(k_{3}i(a_{3}))\\&=
\phi(|k_{3}|,|a_{1}|)\phi(|k_{3}|,|a_{2}|)\phi(|k_{2}|,|a_{1}|)k_{1}i(a_{1})i(S(p(k_{2})p(i(a_{2}))))\otimes\epsilon(k_{3})p(i(a_{3}))\\&=\phi(|k_{3}|,|a_{1}|)\phi(|k_{2}|,|a_{1}|)k_{1}i(a_{1})i(S(p(k_{2}\epsilon(k_{3}))a_{2}))\otimes a_{3}\\&=\phi(|k_{2}|,|a_{1}|)k_{1}i(a_{1})i(S(p(k_{2})a_{2}))\otimes a_{3}=\phi(|k_{2}|,|a_{1}|)k_{1}i(a_{1})i(S(\epsilon(k_{2})a_{2}))\otimes a_{3}\\&=k_{1}\epsilon(k_{2})i(a_{1})i(S(a_{2}))\otimes a_{3}=ki(a_{1}S(a_{2}))\otimes a_{3}=k\epsilon(a_{1})\otimes a_{2}\\&=k\otimes a
\end{split}
\]
and then $f$ is an isomorphism of color Hopf algebras.
\end{proof}
\end{invisible}

The category Act$(\Hc)$ turns out to be equivalent to $\mathrm{Pt}(\Hc)$. This result can be deduced from \cite[Proposition 1.5]{Bohm}. However we include here a direct proof for the reader's sake.

\begin{proposition}
    There is an equivalence of categories between $\mathrm{Act}(\Hc)$ and $\mathrm{Pt}(\Hc)$.
\end{proposition}

\begin{proof}
Define a functor 
\[
F:\mathrm{Pt}(\Hc)\to\mathrm{Act}(\Hc)
\]
as 
\[
\begin{tikzcd}
	H && A \\
	H' && A'
	\arrow[shift left=1, from=1-1, to=1-3, "p"]
	\arrow[shift left=1, from=2-1, to=2-3, "p'"]
	\arrow[shift left=1, from=1-3, to=1-1, "i"]
	\arrow[shift left=1, from=2-3, to=2-1, "i'"]
	\arrow[from=1-1, to=2-1, "f"']
	\arrow[from=1-3, to=2-3, "g"]
\end{tikzcd}\longmapsto
\begin{tikzcd}
	A\otimes\mathrm{Hker}(p) && \mathrm{Hker}(p) \\
	A'\otimes\mathrm{Hker}(p') && \mathrm{Hker}(p')
	\arrow[from=1-1, to=1-3,"\cdot"]
	\arrow[from=1-1, to=2-1, "g\otimes\bar{f}"']
	\arrow[from=1-3, to=2-3, "\bar{f}"]
	\arrow[from=2-1, to=2-3, "\cdot"]
\end{tikzcd}
\]
where the morphism $\cdot$ is defined as in Proposition \ref{prop:fiso} and $\bar{f}$ is induced by the universal property of kernel. Indeed, called $j:\mathrm{Hker}(p)\to H$ the inclusion given by the kernel of $p$, we can compute
\[
p'\circ f\circ j=g\circ p\circ j=g\circ u_{A}\circ\epsilon_{H}\circ j=u_{A'}\circ\epsilon_{H}\circ j=u_{A'}\circ\epsilon_{H'}\circ f\circ j,
\]
so that there exists a unique $\bar{f}:\mathrm{Hker}(p)\to\mathrm{Hker}(p')$ such that $f\circ j=j'\circ\bar{f}$, where $j':\mathrm{Hker}(p')\to H'$ is the inclusion. Clearly $\bar{f}$ is simply the restriction of $f$ to $\mathrm{Hker}(p)$. The diagram on the right effectively commute. Indeed, for all $a\in A$ and $x\in\mathrm{Hker}(p)$, we have 
\[
\begin{split}
g(a)\cdot f(x)&=\phi(|g(a_{2})|,|f(x)|)i'(g(a_{1}))f(x)i'(S(g(a_{2})))=\phi(|a_{2}|,|x|)f(i(a_{1}))f(x)f(i(S(a_{2})))\\&=f(\phi(|a_{2}|,|x|)i(a_{1})xi(S(a_{2})))=f(a\cdot x)
\end{split}
\]
and so $F$ is well-defined. Moreover, define a functor 
\[
G:\mathrm{Act}(\Hc)\to\mathrm{Pt}(\Hc)
\]
as
\[
\begin{tikzcd}
	A\otimes H && H \\
	B\otimes K && K
	\arrow[from=1-1, to=1-3,"\cdot"]
	\arrow[from=1-1, to=2-1, "\alpha\otimes\beta"']
	\arrow[from=1-3, to=2-3, "\beta"]
	\arrow[from=2-1, to=2-3, "\cdot"]
\end{tikzcd}\longmapsto
\begin{tikzcd}
	H\rtimes A && A \\
	K\rtimes B && B
	\arrow[shift left=1, from=1-1, to=1-3, "p_{2}"]
	\arrow[shift left=1, from=2-1, to=2-3, "p_{2}"]
	\arrow[shift left=1, from=1-3, to=1-1, "\iota_2"]
	\arrow[shift left=1, from=2-3, to=2-1, "\iota_2"]
	\arrow[from=1-1, to=2-1, "\beta\otimes\alpha"']
	\arrow[from=1-3, to=2-3, "\alpha"]
\end{tikzcd}
\]
where $p_{2}(h\otimes a)=\epsilon_{H}(h)a$ and $\iota_2(a)=1_{H}\otimes a$ for all $h\in H$ and $a\in A$. Hence $p_{2}\circ\iota_2=\mathrm{Id}_{A}$ and the following relations 
\[
p_{2}(\beta\otimes\alpha)(h\otimes a)=\epsilon_{K}(\beta(h))\alpha(a)=\alpha(\epsilon_{H}(h)a),\ \ \ 1_{K}\otimes\alpha(a)=\beta(1_{H})\otimes\alpha(a)
\]
are satisfied, so that $G$ is well-defined. Hence we have  
\[
\begin{tikzcd}
	(H & A
	\arrow[shift left=1, from=1-1, to=1-2, "p"]
	\arrow[shift left=1, from=1-2, to=1-1, "i"])
\end{tikzcd}\overset{F}\longmapsto
(A\otimes\mathrm{Hker}(p)\overset{\cdot}\longrightarrow\mathrm{Hker}(p))\overset{G}\longmapsto
\begin{tikzcd}
	(\mathrm{Hker}(p)\rtimes A & A
	\arrow[shift left=1, from=1-1, to=1-2, "p_{2}"]
	\arrow[shift left=1, from=1-2, to=1-1, "\iota_2"])
\end{tikzcd}
\]
and we know, by Proposition \ref{prop:fiso}, that $f:\mathrm{Hker}(p)\rtimes A\to H$, $k\otimes a\mapsto ki(a)$ is an isomorphism of color Hopf algebras, so that $(f,\mathrm{Id}_{A})$ is an isomorphism of split epimorphisms, since clearly $f\circ\iota_2=i$ and $p\circ f=p_{2}$ as $p(k)=\epsilon(k)1_{A}$ for $k\in\mathrm{Hker}(p)$. Furthermore, we also have  
\[
(A\otimes H\overset{\cdot}\longrightarrow H)\overset{G}\longmapsto
\begin{tikzcd}
	(H\rtimes A & A
	\arrow[shift left=1, from=1-1, to=1-2, "p_{2}"]
	\arrow[shift left=1, from=1-2, to=1-1, "\iota_2"])
\end{tikzcd}\overset{F}\longmapsto
(A\otimes\mathrm{Hker}(p_{2})\longrightarrow\mathrm{Hker}(p_{2}))
\]
and
\[
\begin{split}
    (\mathrm{Id}\otimes p_{2})\Delta(h\otimes a)&=(\mathrm{Id}\otimes p_{2})(\phi(|h_{2}|,|a_{1}|)h_{1}\otimes a_{1}\otimes h_{2}\otimes a_{2})=\phi(|h_{2}|,|a_{1}|)h_{1}\otimes a_{1}\otimes\epsilon(h_{2})a_{2}\\&=h_{1}\epsilon(h_{2})\otimes a_{1}\otimes a_{2}=h\otimes a_{1}\otimes a_{2},
\end{split}
\]
so that $\mathrm{Hker}(p_{2})$ is given by elements in $H\otimes A^{\mathrm{co}A}=H\otimes\Bbbk1_{A}\cong H$. Therefore, $F$ and $G$ form an equivalence of categories.
\end{proof}

Generalizing \cite[Definition 2.1]{Majid} which is given for (cocommutative) Hopf algebras, we give the following definition for cocommutative color Hopf algebras.

\begin{definition}\label{def:colorHopfcrossed}
    A \textit{color Hopf crossed module} is a triple $(A,H,d)$ where $A$ is a cocommutative color Hopf algebra, $H$ is a cocommutative $A$-module color Hopf algebra and $d:H\to A$ is a morphism of color Hopf algebras such that
\begin{equation}\label{cm1}
    d(a\cdot h)=a\triangleright d(h)=\phi(|a_{2}|,|h|)a_{1}d(h)S(a_{2})\qquad \text{for all}\ h\in H\ \text{and}\ a\in A,
\end{equation}
\begin{equation}\label{cm2}
    d(g)\cdot h=g\triangleright h=\phi(|g_{2}|,|h|)g_{1}hS(g_{2})\qquad \text{for all}\ h,g\in H.
\end{equation}

A morphism $(\alpha,\beta):(A,H,d)\to(A',H',d')$ of color Hopf crossed modules is given by a pair of color Hopf algebra morphisms $\alpha:H\to H'$ and $\beta:A\to A'$ such that $d'\circ\alpha=\beta\circ d$ and $\alpha(a\cdot h)=\beta(a)\cdot\alpha(h)$, for all $a\in A$ and $h\in H$. Denote this category by HXMod($\Hc$).
\end{definition}

Now we show that HXMod($\Hc$) is equivalent to RMG$(\Hc)$. In order to do this we need to know how pullbacks in the category $\Hc$ are made. \\

\noindent\textbf{Pullbacks in $\Hc$}. From the construction of binary products and equalizers in $\Hc$ given in \cite{As} we can easily derive pullbacks in $\Hc$. Let $A,B,C$ be cocommutative color Hopf algebras and $f:A\to C$, $g:B\to C$ be morphisms of color Hopf algebras. The pullback object of $A$ and $B$ over $C$ is given by $\mathrm{Eq}(f\circ\pi_{A},g\circ\pi_{B})$, where $\pi_{A}:=r_{A}\circ(\mathrm{Id}_{A}\otimes\epsilon_{B})$ and $\pi_{B}:=l_{B}\circ(\epsilon_{A}\otimes\mathrm{Id}_{B})$ are the projections of $A\otimes B$, which is the binary product of $A$ and $B$ in $\Hc$. So $A\times_{C}B$ is given by elements $a\otimes b\in A\otimes B$ such that
$$(\mathrm{Id}_{A\otimes B}\otimes(fr_{A}(\mathrm{Id}_{A}\otimes\epsilon_{B})))\Delta_{A\otimes B}(a\otimes b)=(\mathrm{Id}_{A\otimes B}\otimes(gl_{B}(\epsilon_{A}\otimes\mathrm{Id}_{B})))\Delta_{A\otimes B}(a\otimes b).$$ The first member is
\[
(\mathrm{Id}_{A\otimes B}\otimes(fr_{A}(\mathrm{Id}_{A}\otimes\epsilon_{B})))(\phi(|a_{2}|,|b_{1}|)a_{1}\otimes b_{1}\otimes a_{2}\otimes b_{2})=\phi(|a_{2}|,|b_{1}|)a_{1}\otimes b_{1}\otimes f(a_{2}\epsilon(b_{2}))=\phi(|a_{2}|,|b|)a_{1}\otimes b\otimes f(a_{2})
\]
and the second is
\[
(\mathrm{Id}_{A\otimes B}\otimes(gl_{B}(\epsilon_{A}\otimes\mathrm{Id}_{B})))(\phi(|a_{2}|,|b_{1}|)a_{1}\otimes b_{1}\otimes a_{2}\otimes b_{2})=\phi(|a_{2}|,|b_{1}|)a_{1}\otimes b_{1}\otimes g(\epsilon(a_{2})b_{2})=a\otimes b_{1}\otimes g(b_{2}).
\]
Hence $A\times_{C}B$ is given by $a\otimes b\in A\otimes B$ such that $\phi(|a_{2}|,|b|)a_{1}\otimes b\otimes f(a_{2})=a\otimes b_{1}\otimes g(b_{2})$
or, equivalently by applying $\mathrm{Id}_{A}\otimes c_{B,C}$, such that $a_{1}\otimes f(a_{2})\otimes b=
\phi(|b_{1}|,|b_{2}|)a\otimes g(b_{2})\otimes b_{1}=a\otimes g(b_{1})\otimes b_{2}$, where the last equality follows by cocommutativity of $B$.
Therefore, we obtain
\[
A\times_{C}B=\{a\otimes b\in A\otimes B\ |\ a_{1}\otimes f(a_{2})\otimes b=a\otimes g(b_{1})\otimes b_{2}\}.
\]
\begin{invisible}
It is automatically a color Hopf subalgebra of $A\otimes B$ and, considered $\pi_{A}$ and $\pi_{B}$ of before, its elements satisfy $f\circ\pi_{A}=g\circ\pi_{B}$. Let us show explicitly the universal property. Given a cocommutative color Hopf algebra $H$ and $\gamma:H\to B$ and $\phi:H\to A$ two morphisms in $\Hc$ such that $g\circ\gamma=f\circ\phi$, using that $(A\otimes B,\pi_{A},\pi_{B})$ is the binary product of $A$ and $B$ in $\Hc$, there exists a unique morphism $F:=(\phi\otimes\gamma)\circ\Delta_{H}:H\to A\otimes B$ in $\Hc$ such that $\pi_{B}\circ F=\gamma$ and $\pi_{A}\circ F=\phi$. Thus, we only have to prove that $F(h)\in A\times_{C}B$ for all $h\in H$ and this is true since $\phi(h_{1})\otimes f(\phi(h_{2}))\otimes\gamma(h_{3})=
\phi(h_{1})\otimes g(\gamma(h_{2}))\otimes\gamma(h_{3})$.
\end{invisible}
Thus $(A\times_{C}B,\pi_{A},\pi_{B})$ is the pullback of the pair $(f,g)$ in $\Hc$. \medskip

Let us also recall that, since $\Hc$ is action representable by \cite[Proposition 6.3]{As} and so it satisfies the 
condition (SH), Proposition \ref{RMGGrpd} gives us the following result for a reflexive graph in $\Hc$.

\begin{lemma}\label{lemmarefgraph}
    Given a reflexive graph 
\begin{equation}\label{refgrapHc}
\begin{tikzcd}
	A_{1} && A_{0}
	\arrow[shift left=3, from=1-1, to=1-3, "p"]
	\arrow[shift right=3, from=1-1, to=1-3, "\gamma"']
	\arrow["i"{description}, from=1-3, to=1-1]
\end{tikzcd}
\end{equation}
in $\Hc$, the following conditions are equivalent for \eqref{refgrapHc}:
\begin{enumerate}
    \item it is a reflexive-multiplicative graph; \medskip
    \item it is an internal groupoid; \medskip
    \item it satisfies $[\mathrm{Hker}(p),\mathrm{Hker}(\gamma)]_{\mathrm{Huq}}=0$, i.e., $xy=\phi(|x|,|y|)yx$ for all $x\in\mathrm{Hker}(p)$ and $y\in\mathrm{Hker}(\gamma)$.
\end{enumerate}
\end{lemma}

Finally we can show the foretold equivalence between the category $\mathrm{HXMod}(\Hc)$ and the category $\mathrm{RMG}(\Hc)$. This result will allow us to obtain the equivalence between $\mathrm{HXMod}(\Hc)$ and $\mathrm{XMod}(\Hc)$, giving an explicit description to the latter. 

\begin{proposition}\label{eqcat}
    There is an equivalence of categories between $\mathrm{HXMod}(\Hc)$ and $\mathrm{RMG}(\Hc)$. Hence also $\mathrm{HXMod}(\Hc)$ and $\mathrm{Grpd}(\Hc)$ are equivalent categories.
\end{proposition}

\begin{proof}
Given a reflexive graph in $\Hc$ as in \eqref{refgrapHc} which is multiplicative, 
one can consider the morphism of color Hopf algebras $d:=\gamma\circ j:\mathrm{Hker}(p)\to A_{0}$, where $j:\mathrm{Hker}(p)\to A_{1}$ is the canonical inclusion. We know that $\mathrm{Hker}(p)$ is a cocommutative $A_{0}$-module color Hopf algebra by Proposition \ref{prop:fiso} with left $A_{0}$-action given by $a\cdot k=\phi(|a_{2}|,|k|)i(a_{1})ki(S(a_{2}))$ for every $k\in\mathrm{Hker}(p)$ and $a\in A_{0}$. We only have to verify that $d$ satisfies \eqref{cm1} and \eqref{cm2} of Definition \ref{def:colorHopfcrossed}. Thus, for all $a\in A_{0}$ and $k\in\mathrm{Hker}(p)$, we compute
\[
\begin{split}
    d(a\cdot k)&=\gamma(\phi(|a_{2}|,|k|)i(a_{1})ki(S(a_{2})))=\phi(|a_{2}|,|k|)\gamma(i(a_{1}))\gamma(k)\gamma(i(S(a_{2})))=\phi(|a_{2}|,|k|)a_{1}d(k)S(a_{2}).
\end{split}
\]
Since the reflexive graph is multiplicative, by Lemma \ref{lemmarefgraph} we know that $[\mathrm{Hker}(p),\mathrm{Hker}(\gamma)]_{\mathrm{Huq}}=0$, i.e., $xy=\phi(|x|,|y|)yx$ for every $x\in\mathrm{Hker}(p)$ and $y\in\mathrm{Hker}(\gamma)$. But now $S(k_{1})i(\gamma(k_{2}))\in\mathrm{Hker}(\gamma)$ for all $k\in\mathrm{Hker}(p)$, indeed
\[
\begin{split}
(\mathrm{Id}\otimes\gamma)\Delta(S(k_{1})i(\gamma(k_{2})))&=(\mathrm{Id}\otimes\gamma)(\phi(|k_{2}|,|k_{3}|)S(k_{1})i(\gamma(k_{3}))\otimes S(k_{2})i(\gamma(k_{4})))\\&=S(k_{1})i(\gamma(k_{2}))\otimes \gamma(S(k_{3}))\gamma(i(\gamma(k_{4})))\\&=S(k_{1})i(\gamma(k_{2}))\otimes\gamma(S(k_{3})k_{4})\\&=S(k_{1})i(\gamma(k_{2}))\otimes\epsilon(k_{3})1_{A_{0}}\\&=S(k_{1})i(\gamma(k_{2}))\otimes1_{A_{0}}
\end{split}
\]
and then, for all $b,k\in\mathrm{Hker}(p)$, we obtain that
\begin{equation}\label{commHker}
S(k_{1})i(\gamma(k_{2}))b=\phi(|S(k_{1})i(\gamma(k_{2}))|,|b|)bS(k_{1})i(\gamma(k_{2}))=\phi(|k_{1}\otimes k_{2}|,|b|)bS(k_{1})i(\gamma(k_{2})).
\end{equation}
Thus we can compute
\[
\begin{split}
    d(g)\cdot h&=\gamma(g)\cdot h=\phi(|g_{2}|,|h|)i(\gamma(g_{1}))hi(S(\gamma(g_{2})))=\phi(|g_{4}|,|h|)g_{1}S(g_{2})i(\gamma(g_{3}))hi(S(\gamma(g_{4})))\\&\overset{\eqref{commHker}}{=}
\phi(|g_{4}|,|h|)\phi(|g_{2}\otimes g_{3}|,|h|)g_{1}hS(g_{2})i(\gamma(g_{3}))i(\gamma(S(g_{4})))\\&=
\phi(|g_{2}\otimes g_{3}|,|h|)g_{1}hS(g_{2})\epsilon(g_{3})=\phi(|g_{2}|,|h|)g_{1}hS(g_{2})
\end{split}
\]
and then $(A_{0},\mathrm{Hker}(p),d)$ is a color Hopf crossed module. Hence we have obtained a functor $F:\mathrm{RMG}(\Hc)\to\mathrm{HXMod}(\Hc)$ defined as 
\[
\begin{tikzcd}
	A_{1} && A_{0} \\
	A'_{1} && A'_{0}
	\arrow[shift left=3, from=1-1, to=1-3, "p"]
	\arrow[shift right=3, from=1-1, to=1-3, "\gamma"']
	\arrow["i"{description}, from=1-3, to=1-1]
	\arrow[shift left=3, from=2-1, to=2-3, "p'"]
	\arrow[shift right=3, from=2-1, to=2-3, "\gamma'"']
	\arrow["i'"{description}, from=2-3, to=2-1]
	\arrow[from=1-1, to=2-1, "f_{1}"']
	\arrow[from=1-3, to=2-3, "f_{0}"]
\end{tikzcd}
\longmapsto
\begin{tikzcd}
	\mathrm{Hker}(p) && A_{0} \\
	\mathrm{Hker}(p') && A'_{0}
	\arrow[from=1-1, to=1-3, "d"]
	\arrow[from=1-3, to=2-3, "f_{0}"]
	\arrow[from=1-1, to=2-1, "\bar{f_{1}}"']
	\arrow[from=2-1, to=2-3, "d'"']
\end{tikzcd}
\]
where $\bar{f_{1}}$ is the morphism induced by $f_{1}$ using the universal property of the kernel. 
\medskip

Given a color Hopf crossed module $(B,X,d:X\to B)$, we define the reflexive graph
\begin{equation}\label{graphcrossed}
\begin{tikzcd}
	X\rtimes B && B
	\arrow[shift left=3, from=1-1, to=1-3, "p_{2}"]
	\arrow[shift right=3, from=1-1, to=1-3, "p_{1}"']
	\arrow["\iota_{2}"{description}, from=1-3, to=1-1]
\end{tikzcd}
\end{equation}
where $p_{2}(x\otimes b)=\epsilon(x)b$, $p_{1}(x\otimes b)=d(x)b$ for all $x\in X$, $b\in B$ and $\iota_{2}(b)=1_{X}\otimes b$. Note that $p_{1}=m_{B}\circ(d\otimes\mathrm{Id}_{B})$, which is clearly a morphism of graded coalgebras, is also a morphism of algebras. Indeed, given $x,x'\in X$ and $b,b'\in B$, we have 
\[
\begin{split}
p_{1}((x\otimes b)(x'\otimes b'))&=p_{1}(\phi(|b_{2}|,|x'|)x(b_{1}\cdot x')\otimes b_{2}b')=\phi(|b_{2}|,|x'|)d(x(b_{1}\cdot x'))b_{2}b'\\&=\phi(|b_{2}|,|x'|)d(x)d(b_{1}\cdot x')b_{2}b'\overset{\eqref{cm1}}{=}\phi(|b_{3}|,|x'|)\phi(|b_{2}|,|x'|)d(x)b_{1}d(x')S(b_{2})b_{3}b'\\&=\phi(|b_{2}|,|x'|)d(x)b_{1}d(x')\epsilon(b_{2})b'=d(x)bd(x')b'=p_{1}(x\otimes b)p_{1}(x'\otimes b').
\end{split}
\]
Clearly $p_{2}\circ\iota_{2}=p_{1}\circ\iota_{2}=\mathrm{Id}_{B}$. The pullback of the pair $(p_{2},p_{1})$ in $\Hc$ is given by elements $x\otimes b\otimes x'\otimes b'\in(X\rtimes B)\otimes(X\rtimes B)$ such that 
\[
(x\otimes b)_{1}\otimes p_{2}((x\otimes b)_{2})\otimes x'\otimes b'=x\otimes b\otimes p_{1}((x'\otimes b')_{1})\otimes(x'\otimes b')_{2},
\]
i.e., $\phi(|x_{2}|,|b_{1}|)x_{1}\otimes b_{1}\otimes\epsilon(x_{2})b_{2}\otimes x'\otimes b'=\phi(|x'_{2}|,|b'_{1}|)x\otimes b\otimes d(x'_{1})b'_{1}\otimes x'_{2}\otimes b'_{2}$ and then
\[
(X\rtimes B)\times_{B}(X\rtimes B)=\{x\otimes b\otimes x'\otimes b'\in(X\rtimes B)\otimes(X\rtimes B)\ |\ x\otimes b_{1}\otimes b_{2}\otimes x'\otimes b'=\phi(|x'_{2}|,|b'_{1}|)x\otimes b\otimes d(x'_{1})b'_{1}\otimes x'_{2}\otimes b'_{2}\}. 
\]
By applying $\mathrm{Id}_{X}\otimes\epsilon_{B}\otimes\mathrm{Id}_{B}\otimes\mathrm{Id}_{X}\otimes\Delta_{B}$, it follows that 
\begin{equation}\label{eq:1}
x\otimes b\otimes x'\otimes b'_{1}\otimes b'_{2}=\phi(|x'_{2}|,|b'_{1}|)x\otimes\epsilon(b)d(x'_{1})b'_{1}\otimes x'_{2}\otimes b'_{2}\otimes b'_{3}. 
\end{equation}
The reflexive graph \eqref{graphcrossed} is multiplicative with $m:(X\rtimes B)\times_{B}(X\rtimes B)\to X\rtimes B$ defined by $m:=(m_{X}\otimes\mathrm{Id}_{B})\circ(\mathrm{Id}_{X}\otimes\epsilon_{B}\otimes\mathrm{Id}_{X}\otimes\mathrm{Id}_{B})$, i.e.,
\[
m(x\otimes b\otimes x'\otimes b')
=x\epsilon(b)x'\otimes b'\ \text{for all}\ x,x'\in X\ \text{and}\ b,b'\in B.
\]
This is effectively a morphism in $\Hc$. Indeed, $m
$ 
is automatically a morphism of graded coalgebras, so we only have to show that it is a morphism of algebras.
Given $x\otimes b\otimes x'\otimes b'\otimes y\otimes c\otimes y'\otimes c'\in((X\rtimes B)\times_{B}(X\rtimes B))\otimes((X\rtimes B)\times_{B}(X\rtimes B))$ we can compute
\[
\begin{split}
    &m(x\otimes b\otimes x'\otimes b')\cdot m(y\otimes c\otimes y'\otimes c')=(x\epsilon(b)x'\otimes b')(y\epsilon(c)y'\otimes c')=
\\&=\phi(|b'_{2}|,|y\epsilon(c)y'|)x\epsilon(b)x'(b'_{1}\cdot y\epsilon(c)y')\otimes b'_{2}c'=\phi(|c|,|y'|)\phi(|b'_{2}|,|yy'|)x\epsilon(b)x'(b'_{1}\cdot yy')\otimes b'_{2}\epsilon(c)c'\\&=\phi(|c|,|y'|)\phi(|b'_{3}|,|yy'|)\phi(|b'_{2}|,|y|)x\epsilon(b)x'(b'_{1}\cdot y)(b'_{2}\cdot y')\otimes b'_{3}\epsilon(c)c'\\&=\phi(|c|,|y'|)\phi(|b'_{3}|,|yy'|)\phi(|b'_{2}|,|y|)\phi(|x'_{2}|,|b'_{1}\cdot y|)x\epsilon(b)x'_{1}(b'_{1}\cdot y)\epsilon(x'_{2})(b'_{2}\cdot y')\otimes b'_{3}\epsilon(c)c'\\&=\phi(|c|,|y'|)\phi(|b'_{3}|,|yy'|)\phi(|b'_{2}|,|y|)\phi(|x'_{2}\otimes x'_{3}|,|b'_{1}\cdot y|)x\epsilon(b)x'_{1}(b'_{1}\cdot y)S(x'_{2})x'_{3}(b'_{2}\cdot y')\otimes b'_{3}\epsilon(c)c'\\&\overset{\eqref{cm2}}{=}\phi(|c|,|y'|)\phi(|b'_{3}|,|yy'|)\phi(|b'_{2}|,|y|)\phi(|x'_{2}|,|b'_{1}\cdot y|)x\epsilon(b)(d(x'_{1})\cdot(b'_{1}\cdot y))x'_{2}(b'_{2}\cdot y')\otimes b'_{3}\epsilon(c)c'
\end{split}
\]
and
\[
 \begin{split}   
    &m((x\otimes b\otimes x'\otimes b')(y\otimes c\otimes y'\otimes c'))=m(\phi(|x'\otimes b'|,|y\otimes c|)(x\otimes b)(y\otimes c)\otimes(x'\otimes b')(y'\otimes c'))\\&=m(\phi(|x'\otimes b'|,|y\otimes c|)\phi(|b_{2}|,|y|)\phi(|b'_{2}|,|y'|)x(b_{1}\cdot y)\otimes b_{2}c\otimes x'(b'_{1}\cdot y')\otimes b'_{2}c')\\&=\phi(|x'\otimes b'|,|y\otimes c|)\phi(|b_{2}|,|y|)\phi(|b'_{2}|,|y'|)x(b_{1}\cdot y)\epsilon(b_{2}c)x'(b'_{1}\cdot y')\otimes b'_{2}c'\\&=\phi(|x'\otimes b'|,|y|)\phi(|b'_{2}|,|y'|)\phi(|c|,|y'|)x(b\cdot y)x'(b'_{1}\cdot y')\otimes b'_{2}\epsilon(c)c'
\\&=\phi(|c|,|y'|)\phi(|b'_{2}|,|yy'|)\phi(|b'_{1}|,|y|)\phi(|x'|,|y|)x(b\cdot y)x'(b'_{1}\cdot y')\otimes b'_{2}\epsilon(c)c'\\&\overset{\eqref{eq:1}}{=}\phi(|c|,|y'|)\phi(|b'_{3}|,|yy'|)\phi(|b'_{2}|,|y|)\phi(|x'_{2}|,|b'_{1}\cdot y|)x(\epsilon(b)d(x'_{1})b'_{1}\cdot y)x'_{2}(b'_{2}\cdot y')\otimes b'_{3}\epsilon(c)c',
\end{split}
\]
so $m$ is a morphism of algebras. 

Moreover, for all $x\in X$ and $b\in B$, we have
\[
\begin{split}
m(\mathrm{Id}\otimes\iota_{2}p_{2})\Delta_{X\rtimes B}(x\otimes b)&=m(\phi(|x_{2}|,|b_{1}|)x_{1}\otimes b_{1}\otimes\iota_{2}p_{2}(x_{2}\otimes b_{2}))=m(\phi(|x_{2}|,|b_{1}|)x_{1}\otimes b_{1}\otimes\iota_{2}(\epsilon(x_{2})b_{2}))\\&=m(x\otimes b_{1}\otimes\iota_{2}(b_{2}))=m(x\otimes b_{1}\otimes1_{X}\otimes b_{2})=x\epsilon(b_{1})1_{X}\otimes b_{2}=x\otimes b
\end{split}
\]
and also
\[
\begin{split}
m(\iota_{2}p_{1}\otimes\mathrm{Id})\Delta_{X\rtimes B}(x\otimes b)&=m(\phi(|x_{2}|,|b_{1}|)\iota_{2}p_{1}(x_{1}\otimes b_{1})\otimes x_{2}\otimes b_{2})=m(\phi(|x_{2}|,|b_{1}|)\iota_{2}(d(x_{1})b_{1})\otimes x_{2}\otimes b_{2})\\&=m(\phi(|x_{2}|,|b_{1}|)1_{X}\otimes d(x_{1})b_{1}\otimes x_{2}\otimes b_{2})=\phi(|x_{2}|,|b_{1}|)1_{X}\epsilon(d(x_{1})b_{1})x_{2}\otimes b_{2}\\&=\phi(|x_{2}|,|b_{1}|)\epsilon(x_{1})\epsilon(b_{1})x_{2}\otimes b_{2}=x\otimes b,
\end{split}
\]
hence \eqref{graphcrossed} is a reflexive-multiplicative graph.
Therefore, we obtain a functor $G:\mathrm{HXMod}(\Hc)\to\mathrm{RMG}(\Hc)$ defined as
\[
\begin{tikzcd}
	X && B \\
	X' && B'
	\arrow[from=1-1, to=1-3, "d"]
	\arrow[from=1-3, to=2-3, "\beta"]
	\arrow[from=1-1, to=2-1, "\alpha"']
	\arrow[from=2-1, to=2-3, "d'"']
\end{tikzcd}
\longmapsto
\begin{tikzcd}
	X\rtimes B && B \\
	X'\rtimes B' && B'
	\arrow[shift left=3, from=1-1, to=1-3, "p_{2}"]
	\arrow[shift right=3, from=1-1, to=1-3, "p_{1}"']
	\arrow["\iota_{2}"{description}, from=1-3, to=1-1]
	\arrow[shift left=3, from=2-1, to=2-3, "p_{2}"]
	\arrow[shift right=3, from=2-1, to=2-3, "p_{1}"']
	\arrow["\iota_{2}"{description}, from=2-3, to=2-1]
	\arrow[from=1-1, to=2-1, "\alpha\otimes\beta"']
	\arrow[from=1-3, to=2-3, "\beta"]
\end{tikzcd}
\]
and the two functors $F$ and $G$ give rise to an equivalence of categories. 
Indeed, we obtain 
\[
(X\overset{d}\longrightarrow B)\overset{G}\longmapsto
\begin{tikzcd}
	(X\rtimes B && B
	\arrow[shift left=3, from=1-1, to=1-3, "p_{2}"]
	\arrow[shift right=3, from=1-1, to=1-3, "p_{1}"']
	\arrow["\iota_{2}"{description}, from=1-3, to=1-1])
\end{tikzcd}\overset{F}\longmapsto(\mathrm{Hker}(p_{2})\overset{d'}\longrightarrow B)
\]
where $d':=p_{1}\circ j$ and $j:\mathrm{Hker}(p_{2})\to X\rtimes B$ is the inclusion. We know that $\mathrm{Hker}(p_{2})=X\otimes\Bbbk1_{B}\cong X$ 
and so clearly we have an isomorphism of color Hopf crossed modules. Furthermore, we have
\[
\begin{tikzcd}
	(A_{1} && A_{0}
	\arrow[shift left=3, from=1-1, to=1-3, "p"]
	\arrow[shift right=3, from=1-1, to=1-3, "\gamma"']
	\arrow["i"{description}, from=1-3, to=1-1])
\end{tikzcd}\overset{F}\longmapsto
(\mathrm{Hker}(p)\overset{\gamma\circ j'}\longmapsto A_{0})\overset{G}\longmapsto
\begin{tikzcd}
	(\mathrm{Hker}(p)\rtimes A_{0}&& A_{0}
	\arrow[shift left=3, from=1-1, to=1-3, "p_{2}"]
	\arrow[shift right=3, from=1-1, to=1-3, "p_{1}"']
	\arrow["\iota_{2}"{description}, from=1-3, to=1-1])
\end{tikzcd}
\]
where $j':\mathrm{Hker}(p)\to A_{1}$ is the inclusion and $p_{1}(k\otimes a)=\gamma(k)a$ for all $k\in\mathrm{Hker}(p)$ and $a\in A_{0}$. The last reflexive-multiplicative graph is isomorphic to the starting one through the isomorphism $\phi:\mathrm{Hker}(p)\rtimes A_{0}\to A_{1}$, $k\otimes a\mapsto ki(a)$ of Proposition \ref{prop:fiso}. Indeed, for all $k\in\mathrm{Hker}(p)$ and $a\in A_{0}$, we obtain 
\[
p\phi(k\otimes a)=p(ki(a))=p(k)p(i(a))=p(k)a=\epsilon(k)a=p_{2}(k\otimes a), 
\]
i.e., $p\circ\phi=p_{2}$. In addition $\phi\circ\iota_{2}=i$ and 
\[
\gamma(\phi(k\otimes a))=\gamma(ki(a))=\gamma(k)\gamma(i(a))=\gamma(k)a=p_{1}(k\otimes a),
\]
i.e., $\gamma\circ\phi=p_{1}$. Then $\mathrm{HXMod}(\Hc)$ and $\mathrm{RMG}(\Hc)$ are equivalent categories. Moreover, RMG($\Hc$) and Grpd($\Hc$) are isomorphic since $\Hc$ is Mal'tsev \cite{Ca} and so also HXMod$(\Hc)$ and Grpd$(\Hc)$ are equivalent.
\end{proof}


The previous result generalizes \cite[Proposition 5.5]{MG}. Note also that it can be deduced from \cite[Proposition 3.13]{Bohm}, since a color Hopf algebra is precisely a Hopf monoid in the category Vec$_{G}$. 
Finally, we obtain the equivalence between $\mathrm{HXMod}(\Hc)$ and $\mathrm{XMod}(\Hc)$.

\begin{corollary}\label{corollary1}
    There is an equivalence between $\mathrm{HXMod}(\Hc)$ and $\mathrm{XMod}(\Hc)$.
\end{corollary}

\begin{proof}
    For any semi-abelian category $\Cc$, there is an equivalence between Grpd($\Cc$) and XMod($\Cc$), see \cite{Ja}, thus XMod($\Hc$) is equivalent to Grpd($\Hc$). Furthermore, by Proposition \ref{eqcat}, HXMod($\Hc$) is equivalent to 
    Grpd($\Hc$), so 
    the thesis follows.
\end{proof}

As a consequence we also obtain:

\begin{corollary}\label{corollary2}
    The category $\mathrm{XMod}(\Hc)$ and the category $\mathrm{HXMod}(\Hc)$ are semi-abelian.
\end{corollary}

\begin{proof}
    By Proposition \ref{eqcat} and Corollary \ref{corollary1} we know that the categories XMod($\Hc$), HXMod($\Hc$) and Grpd($\Hc$) are all equivalent. Since $\Hc$ is semi-abelian then also Grpd($\Hc$) is semi-abelian (see \cite{MG2} and \cite[Lemma 4.1]{MG3}) and so we are done.
\end{proof}

\begin{remark}
Since the category $\mathrm{HXMod}(\Hc)$ is semi-abelian one could study the category XMod($\mathrm{HXMod}(\Hc)$) of internal crossed modules in the category of color Hopf crossed modules. Note that we have the following equivalences of categories
\[
 \mathrm{XMod}(\mathrm{HXMod}(\Hc))\cong\mathrm{Grpd}(\mathrm{HXMod}(\Hc))\cong\mathrm{Grpd}(\mathrm{Grpd}(\Hc))
\]
where the latter is the category of double internal groupoids in $\Hc$. In \cite[Definitions 2.3 and 2.4]{Sterck} the category of Hopf crossed squares $X^{2}(\mathrm{Hopf}_{\Bbbk,\mathrm{coc}})$ is introduced and in \cite[Corollary 3.3]{Sterck} it is shown that this is equivalent to XMod(HXMod(Hopf$_{\Bbbk,\mathrm{coc}}$)). Then it would be interesting to extend the previous result to $\Hc$ introducing the notion of color Hopf crossed squares. We leave this as a possible future project. 
\end{remark}

\section{Simplicial color Hopf algebras and color Hopf crossed modules}

In \cite[Theorem 5.5]{Emir} it is shown that the category of crossed modules of cocommutative Hopf algebras is equivalent to the category of simplicial cocommutative Hopf algebras with Moore complex of length one. We want to extend this result to the case of cocommutative color Hopf algebras. We start by giving the following two definitions which employ the notions given in the preliminary section.

\begin{definition}
A simplicial cocommutative color Hopf algebra $\mathcal{H}$ is a simplicial object in the category $\Hc$. In other words, it is given by a collection of cocommutative color Hopf algebras $H_{n}$, with $n\in\mathbb{N}$, together with color Hopf algebra maps 
\[
d^{n}_{i}:H_{n}\to H_{n-1},\ \text{for}\ 0\leq i\leq n\ \ \text{and}\ \ s^{n+1}_{j}:H_{n}\to H_{n+1},\ \text{for}\ 0\leq j\leq n
\]
that satisfy the simplicial identities:
\begin{enumerate}
    \item[1)] $d^{n-1}_{i}\circ d^{n}_{j}=d^{n-1}_{j-1}\circ d^{n}_{i}$ if $i<j$; \medskip
    \item[2)] $s^{n+1}_{i}\circ s^{n}_{j}=s^{n+1}_{j+1}\circ s^{n}_{i}$ if $i\leq j$; \medskip
    \item[3)] $d^{n}_{i}\circ s^{n}_{j}=s^{n-1}_{j-1}\circ d^{n-1}_{i}$ if $i<j$,\ $d^{n}_{j}\circ s^{n}_{j}=d^{n}_{j+1}\circ s^{n}_{j}=\mathrm{Id}$,\ $d^{n}_{i}\circ s^{n}_{j}=s^{n-1}_{j}\circ d^{n-1}_{i-1}$ if $i>j+1$.
\end{enumerate}
A simplicial cocommutative color Hopf algebra can be represented by the following diagram
\begin{equation}\label{simplicolor}
\begin{tikzcd}
	\mathcal{H}: && H_{3} & H_{2} & H_{1} & H_{0}
	\arrow["d_{0}"{description},from=1-5, to=1-6]
	\arrow[shift left=2, from=1-5, to=1-6, "d_{1}"]
	\arrow[shift left=2, from=1-6, to=1-5,"s_{0}"]
	\arrow["d_{0}"{description}, from=1-4, to=1-5]
	\arrow["d_{1}"{description}, shift left=3, from=1-4, to=1-5]
	\arrow[shift left=5, from=1-4, to=1-5,"d_{2}"]
	\arrow["s_{0}"{description}, shift left=3, from=1-5, to=1-4]
	\arrow[shift left=5, from=1-5, to=1-4, "s_{1}"]
	\arrow[shift left=2, from=1-3, to=1-4]
	\arrow[shift left=6, from=1-3, to=1-4]
	\arrow[shift left=4, from=1-3, to=1-4]
	\arrow[from=1-3, to=1-4]
	\arrow[shift left=2, from=1-4, to=1-3]
	\arrow[shift left=4, from=1-4, to=1-3]
	\arrow[shift left=6, from=1-4, to=1-3]
	\arrow[dotted, no head, from=1-1, to=1-3]
\end{tikzcd}
\end{equation}
where we omit the upper indexes for the maps $d_{i}$ and $s_{i}$. 
\end{definition}

\begin{definition}
Given a simplicial cocommutative color Hopf algebra $\mathcal{H}$, 
the chain complex $(M(\mathcal{H})_{\bullet},\partial_{\bullet})$ is given by:
\begin{enumerate}
    \item $M(\mathcal{H})_{n}=\{0\}$ for $n<0$ and $M(\mathcal{H})_{0}=H_{0}$, \medskip
    \item $M(\mathcal{H})_{n}=\bigcap_{i=0}^{n-1}{\mathrm{Hker}(d^{n}_{i})}$ for $n\geq1$, \medskip
    \item $\partial_{n}:M(\mathcal{H})_{n}\to M(\mathcal{H})_{n-1}$ is the restriction of $d^{n}_{n}$ to $M(\mathcal{H})_{n}$ for $n\geq1$ (and the zero morphism for $n\leq0$),
\end{enumerate}
which is called the \textit{Moore complex} of $\mathcal{H}$. We say that it has length $m$ if $M(\mathcal{H})_{i}$ is the zero object $\Bbbk$ for all $i>m$.
\end{definition}

Note that, given a simplicial cocommutative color Hopf algebra $\mathcal{H}$ as in \eqref{simplicolor}, we have a splitting of morphisms \begin{tikzcd}
	H_{n} & H_{n-1}
	\arrow[shift left=1, from=1-1, to=1-2, "d^{n}_{i}"]
	\arrow[shift left=1, from=1-2, to=1-1,"s^{n}_{i}"]
\end{tikzcd}
in $\Hc$ for all $n\geq1$, by 3) of the simplicial identities. Hence, from Proposition \ref{prop:fiso} 
we immediately obtain the following result:

\begin{proposition}\label{prop:fiso2}
In a simplicial cocommutative color Hopf algebra $\mathcal{H}$, there exists an action
\[
\cdot_{i}: H_{n-1}\otimes\mathrm{Hker}(d^{n}_{i})\to\mathrm{Hker}(d^{n}_{i}),\ a\otimes k\mapsto\phi(|a_{2}|,|k|)s^{n}_{i}(a_{1})ks^{n}_{i}(S(a_{2}))
\]
for all $0\leq i\leq n-1$ and $n\geq1$ which makes $\mathrm{Hker}(d^{n}_{i})$ a cocommutative $H_{n-1}$-module color Hopf algebra. Moreover, we have an isomorphism of color Hopf algebras
\[
H_{n}\cong\mathrm{Hker}(d^{n}_{i})\rtimes_{i}H_{n-1} 
\]
for all $1\leq n\in\mathbb{N}$ and $0\leq i\leq n-1$.
\end{proposition}

\begin{remark}
    Explicitly, in Proposition \ref{prop:fiso2}, the isomorphism is given by $f:\mathrm{Hker}(d^{n}_{i})\rtimes H_{n-1}\to H_{n}$, $k\otimes a\mapsto ks^{n}_{i}(a)$ with inverse given by $g:H_{n}\to\mathrm{Hker}(d^{n}_{i})\rtimes H_{n-1}$, $h\mapsto h_{1}s^{n}_{i}d^{n}_{i}(S(h_{2}))\otimes d^{n}_{i}(h_{3})$, i.e., $g=(f^{n}_{i}\otimes d^{n}_{i})\circ\Delta_{H_{n}}$ with $f^{n}_{i}:H_{n}\to\mathrm{Hker}(d^{n}_{i})\subseteq H_{n}$, $x\mapsto x_{1}s^{n}_{i}d^{n}_{i}(S(x_{2}))$. Recall that $a\cdot_{i}k=\xi_{H_{n}}(s^{n}_{i}(a)\otimes k)=s^{n}_{i}(a)\triangleright k$. Note that, given $x\in H_{n}$ and considered 
\[  
    f^{n}_{i+1}f^{n}_{i}(x)=f^{n}_{i+1}(x_{1}s^{n}_{i}d^{n}_{i}(S(x_{2})))\overset{(*)}{=}x_{1}s^{n}_{i}d^{n}_{i}(S(x_{2}))s^{n}_{i+1}d^{n}_{i+1}(s^{n}_{i}d^{n}_{i}(x_{3})S(x_{4})),
\]    
where $(*)$ follows by cocommutativity, then
\[
\begin{split}
d^{n}_{i}(f^{n}_{i+1}f^{n}_{i}(x))&=d^{n}_{i}(x_{1})d^{n}_{i}s^{n}_{i}d^{n}_{i}(S(x_{2}))d^{n}_{i}s^{n}_{i+1}d^{n}_{i+1}s^{n}_{i}d^{n}_{i}(x_{3})d^{n}_{i}s^{n}_{i+1}d^{n}_{i+1}(S(x_{4}))\\&=d^{n}_{i}(x_{1})d^{n}_{i}(S(x_{2}))d^{n}_{i}s^{n}_{i+1}d^{n}_{i}(x_{3})s^{n-1}_{i}d^{n-1}_{i}d^{n}_{i+1}(S(x_{4}))\\&=d^{n}_{i}(x_{1}S(x_{2}))s^{n-1}_{i}d^{n-1}_{i}d^{n}_{i}(x_{3})s^{n-1}_{i}d^{n-1}_{i}d^{n}_{i}(S(x_{4}))\\&=\epsilon(x_{1})s^{n-1}_{i}d^{n-1}_{i}d^{n}_{i}(x_{2}S(x_{3}))=\epsilon(x_{1})\epsilon(x_{2})=\epsilon(x),
\end{split}
\]
and then $f^{n}_{i+1}f^{n}_{i}(x)\in\mathrm{Hker}(d^{n}_{i+1})\cap\mathrm{Hker}(d^{n}_{i})$. Hence $f^{n}:=f^{n}_{n-1}\circ f^{n}_{n-2}\circ\cdots\circ f^{n}_{1}\circ f^{n}_{0}:H_{n}\to\bigcap_{i=0}^{n-1}{\mathrm{Hker}(d^{n}_{i})}=M(\mathcal{H})_{n}$, where we identify $f^{n}_{i}$ with $j^{n}_{i}\circ f^{n}_{i}$ and $j^{n}_{i}:\mathrm{Hker}(d^{n}_{i})\to H_{n}$ is the canonical inclusion. In particular, if $n=2$, $f^{2}:=f^{2}_{1}\circ f^{2}_{0}:H_{2}\to M(\mathcal{H})_{2}$, where $f^{2}_{0}:H_{2}\to\mathrm{Hker}(d^{2}_{0})$, $x\mapsto x_{1}s^{2}_{0}d^{2}_{0}(S(x_{2}))$ and $f^{2}_{1}:H_{2}\to\mathrm{Hker}(d^{2}_{1})$, $x\mapsto x_{1}s^{2}_{1}d^{2}_{1}(S(x_{2}))$. We will use the morphism $f^{2}$ in the following.
\end{remark}

Note that, given a simplicial cocommutative color Hopf algebra $\mathcal{H}$ as in \eqref{simplicolor}, we can obtain a new simplicial cocommutative color Hopf algebra by considering kernels, where the first three components are given by
\begin{equation}\label{simplicolor2}
\begin{tikzcd}
	\mathrm{Hker}(d^{3}_{0})\subseteq H_{3} & \mathrm{Hker}(d^{2}_{0})\subseteq H_{2} & \mathrm{Hker}(d^{1}_{0})\subseteq H_{1}
	\arrow["d_{1}"{description}, from=1-2, to=1-3]
	\arrow[shift left=2, from=1-2, to=1-3,"d_{2}"]
	\arrow[shift left=2, from=1-3, to=1-2,"s_{1}"]
	\arrow["d_{1}"{description}, from=1-1, to=1-2]
	\arrow["d_{2}"{description}, shift left=3, from=1-1, to=1-2]
	\arrow[shift left=5, from=1-1, to=1-2,"d_{3}"]
	\arrow["s_{1}"{description}, shift left=3, from=1-2, to=1-1]
	\arrow[shift left=5, from=1-2, to=1-1, "s_{2}"]
\end{tikzcd}
\end{equation}
and one can repeat the same process taking kernels in \eqref{simplicolor2} and obtaining a new simplicial cocommutative color Hopf algebra with first three components
\begin{equation}\label{simplicolor3}
\begin{tikzcd}
	\mathrm{Hker}(d^{4}_{0})\cap\mathrm{Hker}(d^{4}_{1})\subseteq H_{4} & \mathrm{Hker}(d^{3}_{0})\cap\mathrm{Hker}(d^{3}_{1})\subseteq H_{3} & \mathrm{Hker}(d^{2}_{0})\cap\mathrm{Hker}(d^{2}_{1})\subseteq H_{2}.
	\arrow["d_{2}"{description}, from=1-2, to=1-3]
	\arrow[shift left=2, from=1-2, to=1-3, "d_{3}"]
	\arrow[shift left=2, from=1-3, to=1-2, "s_{2}"]
	\arrow["d_{2}"{description}, from=1-1, to=1-2]
	\arrow["d_{3}"{description}, shift left=3, from=1-1, to=1-2]
	\arrow[shift left=5, from=1-1, to=1-2, "d_{4}"]
	\arrow["s_{2}"{description}, shift left=3, from=1-2, to=1-1]
	\arrow[shift left=5, from=1-2, to=1-1, "s_{3}"]
\end{tikzcd}
\end{equation}
By applying Proposition \ref{prop:fiso2} to \eqref{simplicolor} we obtain
\[
H_{1}\cong\mathrm{Hker}(d^{1}_{0})\rtimes_{0}H_{0}=M(\mathcal{H})_{1}\rtimes_{0}M(\mathcal{H})_{0}
\]
and 
\[
H_{2}\cong\mathrm{Hker}(d^{2}_{0})\rtimes_{0} H_{1}\cong\mathrm{Hker}(d^{2}_{0})\rtimes_{0}(M(\mathcal{H})_{1}\rtimes_{0}M(\mathcal{H})_{0}).
\]
Moreover, if we apply Proposition \ref{prop:fiso2} to \eqref{simplicolor2} we further obtain 
\[
\begin{split}
    \mathrm{Hker}(d^{2}_{0})\cong\mathrm{Hker}(d^{2}_{1})|_{\mathrm{Hker}(d^{2}_{0})}\rtimes_{1}\mathrm{Hker}(d^{1}_{0})=(\mathrm{Hker}(d^{2}_{1})\cap\mathrm{Hker}(d^{2}_{0}))\rtimes_{1}\mathrm{Hker}(d^{1}_{0})=(M(\mathcal{H})_{2}\rtimes_{1} M(\mathcal{H})_{1})
\end{split}
\]
and then
\[
\begin{split}
H_{2}&\cong
(M(\mathcal{H})_{2}\rtimes_{1} M(\mathcal{H})_{1})\rtimes_{0}(M(\mathcal{H})_{1}\rtimes_{0} M(\mathcal{H})_{0}).
\end{split}
\]
Let us also analyze the case of $H_{3}$, using \eqref{simplicolor2} and \eqref{simplicolor3}. Since 
\[
\begin{split}
\mathrm{Hker}(d^{3}_{0})\cap\mathrm{Hker}(d^{3}_{1})&\cong(\mathrm{Hker}(d^{3}_{2})|_{(\mathrm{Hker}(d^{3}_{0})\cap\mathrm{Hker}(d^{3}_{1}))})\rtimes_{2}(\mathrm{Hker}(d^{2}_{0})\cap\mathrm{Hker}(d^{2}_{1}))\\&=(\mathrm{Hker}(d^{3}_{0})\cap\mathrm{Hker}(d^{3}_{1})\cap\mathrm{Hker}(d^{3}_{2}))\rtimes_{2}(\mathrm{Hker}(d^{2}_{0})\cap\mathrm{Hker}(d^{2}_{1}))\\&=M(\mathcal{H})_{3}\rtimes_{2}M(\mathcal{H})_{2},
\end{split}
\]
we obtain
\[
\begin{split}
\mathrm{Hker}(d^{3}_{0})&\cong(\mathrm{Hker}(d^{3}_{1})|_{\mathrm{Hker}(d^{3}_{0})})\rtimes_{1}\mathrm{Hker}(d^{2}_{0})=(\mathrm{Hker}(d^{3}_{0})\cap\mathrm{Hker}(d^{3}_{1}))\rtimes_{1}\mathrm{Hker}(d^{2}_{0})\\&\cong
(M(\mathcal{H})_{3}\rtimes_{2}M(\mathcal{H})_{2})\rtimes_{1}(M(\mathcal{H})_{2}\rtimes_{1} M(\mathcal{H})_{1})
\end{split}
\]
and then that
\[
H_{3}\cong\mathrm{Hker}(d^{3}_{0})\rtimes_{0} H_{2}\cong((M(\mathcal{H})_{3}\rtimes_{2} M(\mathcal{H})_{2})\rtimes_{1}(M(\mathcal{H})_{2}\rtimes_{1} M(\mathcal{H})_{1}))\rtimes_{0}((M(\mathcal{H})_{2}\rtimes_{1} M(\mathcal{H})_{1})\rtimes_{0}(M(\mathcal{H})_{1}\rtimes_{0} M(\mathcal{H})_{0})).
\]
Clearly we can iterate this process arriving at a decomposition of $H_{n}$ for all $n\geq0$, analogously to \cite[Theorem 3.18]{Emir}.

In order to show the equivalence between color Hopf crossed modules and simplicial cocommutative color Hopf algebras with Moore complex of length one we prove some intermediate results.

\begin{lemma}\label{mapf2}
Given $x\in H_{1}$ and $y\in M(\mathcal{H})_{1}$ we obtain 
\[
M(\mathcal{H})_{2}\ni f^{2}(s^{2}_{0}(x)\triangleright s^{2}_{1}(y))=\phi(|x_{2}|,|y_{1}|)(s^{2}_{0}(x_{1})\triangleright s^{2}_{1}(y_{1}))S(s^{2}_{1}(x_{2})\triangleright s^{2}_{1}(y_{2}))
\]
\end{lemma}

\begin{proof}
Given $x\in H_{1}$ and $y\in M(\mathcal{H})_{1}$ we can compute
\[
\begin{split}
    f^{2}_{0}(s^{2}_{0}(x)\triangleright s^{2}_{1}(y))&=(s^{2}_{0}(x)\triangleright s^{2}_{1}(y))_{1}s^{2}_{0}d^{2}_{0}(S((s^{2}_{0}(x)\triangleright s^{2}_{1}(y))_{2}))\\&\overset{\eqref{xiofcoalg}}{=}\phi(|x_{2}|,|y_{1}|)(s^{2}_{0}(x_{1})\triangleright s^{2}_{1}(y_{1}))s^{2}_{0}d^{2}_{0}(S(s^{2}_{0}(x_{2})\triangleright s^{2}_{1}(y_{2})))\\&=\phi(|x_{2}|,|y_{1}|)(s^{2}_{0}(x_{1})\triangleright s^{2}_{1}(y_{1}))S(s^{2}_{0}d^{2}_{0}(s^{2}_{0}(x_{2})\triangleright s^{2}_{1}(y_{2})))\\&=\phi(|x_{2}|,|y_{1}|)(s^{2}_{0}(x_{1})\triangleright s^{2}_{1}(y_{1}))S(s^{2}_{0}d^{2}_{0}s^{2}_{0}(x_{2})\triangleright s^{2}_{0}d^{2}_{0}s^{2}_{1}(y_{2}))\\&=\phi(|x_{2}|,|y_{1}|)(s^{2}_{0}(x_{1})\triangleright s^{2}_{1}(y_{1}))S(s^{2}_{0}(x_{2})\triangleright s^{2}_{0}s^{1}_{0}d^{1}_{0}(y_{2}))\\&\overset{(*)}{=}\phi(|x_{2}|,|y_{1}|)(s^{2}_{0}(x_{1})\triangleright s^{2}_{1}(y_{1}))S(s^{2}_{0}(x_{2})\triangleright\epsilon(y_{2})1_{H_{2}})\\&\overset{\eqref{xiantipode}}{=}\phi(|x_{2}|,|y_{1}|)(s^{2}_{0}(x_{1})\triangleright s^{2}_{1}(y_{1}))(s^{2}_{0}(x_{2})\triangleright\epsilon(y_{2})1_{H_{2}})\\&\overset{\eqref{xiDelta}}{=}s^{2}_{0}(x)\triangleright s^{2}_{1}(y)
\end{split}
\]
where $(*)$ follows since $y\in\mathrm{Hker}(d^{1}_{0})$, i.e., $y_{1}\otimes d^{1}_{0}(y_{2})=y_{1}\otimes\epsilon(y_{2})1_{H_{0}}$ and then 
\[
s^{2}_{1}(y_{1})\otimes s^{2}_{0}s^{1}_{0}d^{1}_{0}(y_{2})=s^{2}_{1}(y_{1})\otimes s^{2}_{0}s^{1}_{0}(\epsilon(y_{2})1_{H_{0}})=s^{2}_{1}(y_{1})\otimes\epsilon(y_{2})1_{H_{2}}.
\]
Hence we obtain
\[
\begin{split}
    f^{2}(s^{2}_{0}(x)\triangleright s^{2}_{1}(y))&=f^{2}_{1}f^{2}_{0}(s^{2}_{0}(x)\triangleright s^{2}_{1}(y))=
    f^{2}_{1}(s^{2}_{0}(x)\triangleright s^{2}_{1}(y))\\&=(s^{2}_{0}(x)\triangleright s^{2}_{1}(y))_{1}s^{2}_{1}d^{2}_{1}(S((s^{2}_{0}(x)\triangleright s^{2}_{1}(y))_{2}))\\&\overset{\eqref{xiofcoalg}}{=}\phi(|x_{2}|,|y_{1}|)(s^{2}_{0}(x_{1})\triangleright s^{2}_{1}(y_{1}))s^{2}_{1}d^{2}_{1}(S(s^{2}_{0}(x_{2})\triangleright s^{2}_{1}(y_{2})))\\&=\phi(|x_{2}|,|y_{1}|)(s^{2}_{0}(x_{1})\triangleright s^{2}_{1}(y_{1}))S(s^{2}_{1}d^{2}_{1}s^{2}_{0}(x_{2})\triangleright s^{2}_{1}d^{2}_{1}s^{2}_{1}(y_{2}))\\&=\phi(|x_{2}|,|y_{1}|)(s^{2}_{0}(x_{1})\triangleright s^{2}_{1}(y_{1}))S(s^{2}_{1}(x_{2})\triangleright s^{2}_{1}(y_{2}))
\end{split}
\]
and then the thesis follows.
\end{proof}

\begin{proposition}\label{functor1}
Let $\mathcal{H}$ be a simplicial cocommutative color Hopf algebra with Moore complex of length one. We obtain the color Hopf crossed module $(H_{0},M(\mathcal{H})_{1}=\mathrm{Hker}(d^{1}_{0}),\partial_{1}:M(\mathcal{H})_{1}\to H_{0})$, where $M(\mathcal{H})_{1}$ is a cocommutative $H_{0}$-module color Hopf algebra through
\[
\cdot:H_{0}\otimes M(\mathcal{H})_{1}\to M(\mathcal{H})_{1},\ k\otimes x\mapsto
\phi(|k_{2}|,|x|)s^{1}_{0}(k_{1})xs^{1}_{0}(S(k_{2}))
\]
and $\partial_{1}:=d^{1}_{1}|_{M(\mathcal{H})_{1}}:M(\mathcal{H})_{1}\to H_{0}$.
\end{proposition}

\begin{proof}
    By Proposition \ref{prop:fiso2} we already know that 
    $M(\mathcal{H})_{1}$ is a cocommutative $H_{0}$-module color Hopf algebra through the map $\cdot$ defined, so we only have to show that the morphism of color Hopf algebras $\partial_{1}$ satisfies \eqref{cm1} and \eqref{cm2}. For all $k\in H_{0}$ and $x\in M(\mathcal{H})_{1}$ we can compute
\[
\begin{split}
    \partial_{1}(k\cdot x)&=d^{1}_{1}(\phi(|k_{2}|,|x|)s^{1}_{0}(k_{1})xs^{1}_{0}(S(k_{2})))=\phi(|k_{2}|,|x|)d^{1}_{1}s^{1}_{0}(k_{1})d^{1}_{1}(x)d^{1}_{1}s^{1}_{0}(S(k_{2}))
    =\phi(|k_{2}|,|x|)k_{1}\partial_{1}(x)S(k_{2})
\end{split}
\]
and then \eqref{cm1} is satisfied. Moreover, for all $x,y\in M(\mathcal{H})_{1}$, we have
\[
\begin{split}
    \partial_{1}(x)\cdot y&=d^{1}_{1}(x)\cdot y=\phi(|x_{2}|,|y|)s^{1}_{0}d^{1}_{1}(x_{1})ys^{1}_{0}(S(d^{1}_{1}(x_{2})))
=\phi(|x_{2}|,|y|)d^{2}_{2}s^{2}_{0}(x_{1})yS(d^{2}_{2}s^{2}_{0}(x_{2}))=d^{2}_{2}s^{2}_{0}(x)\triangleright y
\end{split}
\]
so that, in order to prove that also \eqref{cm2} holds true, we have to show $d^{2}_{2}s^{2}_{0}(x)\triangleright y=x\triangleright y$ for all $x,y\in M(\mathcal{H})_{1}$. 
Hence, using Lemma \ref{mapf2}, we obtain
\[
\begin{split}
    d^{2}_{2}f^{2}(s^{2}_{0}(x)\triangleright s^{2}_{1}(y))&=\phi(|x_{2}|,|y_{1}|)d^{2}_{2}(s^{2}_{0}(x_{1})\triangleright s^{2}_{1}(y_{1}))d^{2}_{2}S(s^{2}_{1}(x_{2})\triangleright s^{2}_{1}(y_{2}))\\&=\phi(|x_{2}|,|y_{1}|)(d^{2}_{2}s^{2}_{0}(x_{1})\triangleright d^{2}_{2}s^{2}_{1}(y_{1}))S(d^{2}_{2}s^{2}_{1}(x_{2})\triangleright d^{2}_{2}s^{2}_{1}(y_{2}))\\&=\phi(|x_{2}|,|y_{1}|)(d^{2}_{2}s^{2}_{0}(x_{1})\triangleright y_{1})S(x_{2}\triangleright y_{2})\\&\overset{\eqref{xiantipode}}{=}\phi(|x_{2}|,|y_{1}|)(d^{2}_{2}s^{2}_{0}(x_{1})\triangleright y_{1})(x_{2}\triangleright S(y_{2})).
\end{split}
\]
But we are under the assumption that the Moore complex of $\mathcal{H}$ has length one and then $M(\mathcal{H})_{2}=\Bbbk1_{H_{2}}$, so that $\partial_{2}=d^{2}_{2}|_{M(\mathcal{H})_{2}}:M(\mathcal{H})_{2}\to M(\mathcal{H})_{1}$ is the zero morphism. Therefore, we obtain
\[
\phi(|x_{2}|,|y_{1}|)(d^{2}_{2}s^{2}_{0}(x_{1})\triangleright y_{1})(x_{2}\triangleright S(y_{2}))=\epsilon(x)\epsilon(y)1_{H_{1}}.
\]
Finally, we can compute
\[
\begin{split}
    x\triangleright y&=
\phi(|x_{2}|,|y_{1}|)\epsilon(x_{1})\epsilon(y_{1})(x_{2}\triangleright y_{2})\\&=\phi(|x_{3}|,|y_{1}\otimes y_{2}|)\phi(|x_{2}|,|y_{1}|)(d^{2}_{2}s^{2}_{0}(x_{1})\triangleright y_{1})(x_{2}\triangleright S(y_{2}))(x_{3}\triangleright y_{3})\\&\overset{\eqref{xiDelta}}{=}\phi(|x_{2}|,|y_{1}|)(d^{2}_{2}s^{2}_{0}(x_{1})\triangleright y_{1})(x_{2}\triangleright\epsilon(y_{2})1_{H_{1}})\\&=\phi(|x_{2}|,|y|)(d^{2}_{2}s^{2}_{0}(x_{1})\triangleright y)(x_{2}\triangleright1_{H_{1}})=\phi(|x_{2}|,|y|)(d^{2}_{2}s^{2}_{0}(x_{1})\triangleright y)\epsilon(x_{2})\\&=d^{2}_{2}s^{2}_{0}(x)\triangleright y
\end{split}
\]
and so also \eqref{cm2} is satisfied.
\end{proof}

We can also obtain a simplicial cocommutative color Hopf algebra with Moore complex of length one starting from a color Hopf crossed module. 

\begin{remark}\label{oss:2cosk}
A $2$-truncated simplicial cocommutative color Hopf algebra is given by
\[
\begin{tikzcd}
	H_{2} & H_{1} & H_{0}
	\arrow["d_{0}"{description}, from=1-2, to=1-3]
	\arrow[shift left=2, from=1-2, to=1-3,"d_{1}"]
	\arrow[shift left=2, from=1-3, to=1-2,"s_{0}"]
	\arrow["d_{0}"{description}, from=1-1, to=1-2]
	\arrow["d_{1}"{description}, shift left=3, from=1-1, to=1-2]
	\arrow[shift left=5, from=1-1, to=1-2,"d_{2}"]
	\arrow["s_{0}"{description}, shift left=3, from=1-2, to=1-1]
	\arrow[shift left=5, from=1-2, to=1-1, "s_{1}"]
\end{tikzcd}
\]
where $H_{0}$, $H_{1}$ and $H_{2}$ are cocommutative color Hopf algebras and the maps $d_{i}$, $s_{i}$ satisfy the simplicial identities. Recall that from it we can obtain a simplicial cocommutative color Hopf algebra by using the 2-coskeleton functor, as explained in the preliminary section.
\end{remark}

\begin{proposition}\label{functor2}
    Given a color Hopf crossed module $(A,H,d:H\to A)$ we obtain a 2-truncated simplicial cocommutative color Hopf algebra
\begin{equation}\label{simplicolorprop}
\begin{tikzcd}
	H\rtimes_{*}(H\rtimes A) & H\rtimes A & A
	\arrow["d_{0}"{description}, from=1-2, to=1-3]
	\arrow[shift left=2, from=1-2, to=1-3,"d_{1}"]
	\arrow[shift left=2, from=1-3, to=1-2,"s_{0}"]
	\arrow["d_{0}"{description}, from=1-1, to=1-2]
	\arrow["d_{1}"{description}, shift left=3, from=1-1, to=1-2]
	\arrow[shift left=5, from=1-1, to=1-2,"d_{2}"]
	\arrow["s_{0}"{description}, shift left=3, from=1-2, to=1-1]
	\arrow[shift left=5, from=1-2, to=1-1, "s_{1}"]
\end{tikzcd}
\end{equation}
and then, as said in Remark \ref{oss:2cosk}, a simplicial cocommutative color Hopf algebra. The latter has Moore complex of length one.
\end{proposition}

\begin{proof}
Since $H$ is a cocommutative $A$-module color Hopf algebra we can make the semi-direct product $H\rtimes A$. Denote the $A$-action of $H$ by $\cdot:A\otimes H\to H$ and define $H_{0}:=A$ and $H_{1}:=H\rtimes A$. Moreover, define morphisms $d^{1}_{0},d^{1}_{1}:H_{1}\to H_{0}$ and $s^{1}_{0}:H_{0}\to H_{1}$ in $\Hc$ as 
\[
d^{1}_{0}(h\otimes a):=\epsilon(h)a,\ d^{1}_{1}(h\otimes a):=d(h)a,\ s^{1}_{0}(a):=1_{H}\otimes a,
\]
so that, clearly, $d^{1}_{0}\circ s^{1}_{0}=d^{1}_{1}\circ s^{1}_{0}=\mathrm{Id}_{A}$. Note that these morphisms are exactly those defined in \eqref{graphcrossed}. We can define a $(H\rtimes A)$-action on $H$ as
\[
*:(H\rtimes A)\otimes H\to H,\ (h\otimes a)\otimes h'\mapsto(d(h)a)\cdot h',
\]
which makes $H$ a cocommutative $(H\rtimes A)$-module color Hopf algebra.
Indeed, the morphism $*$ is an action since $(1_{H}\otimes 1_{A})* h'=(d(1_{H})1_{A})\cdot h'=1_{A}\cdot h'=h'$
and, since $d$ is a morphism of algebras, we also have 
\[
\begin{split}
((h'\otimes a')(h\otimes a))* h''&=(\phi(|a'_{2}|,|h|)h'(a'_{1}\cdot h)\otimes a'_{2}a)* h''=(d(\phi(|a'_{2}|,|h|)h'(a'_{1}\cdot h))a'_{2}a)\cdot h''\\&=(\phi(|a'_{2}|,|h|)d(h')d(a'_{1}\cdot h)a'_{2}a)\cdot h''\overset{\eqref{cm1}}{=}(\phi(|a'_{3}|,|h|)\phi(|a'_{2}|,|h|)d(h')a'_{1}d(h)S(a'_{2})a'_{3}a)\cdot h''\\&=(\phi(|a'_{2}|,|h|)d(h')a'_{1}d(h)\epsilon(a'_{2})a)\cdot h''=(d(h')a'd(h)a)\cdot h''=(d(h')a')\cdot((d(h)a)\cdot h'')\\&=
(h'\otimes a')*((h\otimes a)* h'').
\end{split}
\]
Moreover, clearly $*$ makes $H$ a cocommutative $(H\rtimes A)$-module color Hopf algebra since the compatibility conditions are satisfied by the action $\cdot:A\otimes H\to H$.

\begin{invisible}
Indeed
\[
(h\otimes a)*1_{H}=d(h)a\cdot1_{H}=\epsilon(d(h)a)1_{H}=\epsilon(h)\epsilon(a)1_{H}=\epsilon_{H\rtimes A}(h\otimes a)1_{H}
\]
and
\[
\epsilon((h\otimes a)* h')=\epsilon(d(h)a\cdot h')=\epsilon(d(h)a)\epsilon(h')=\epsilon_{H\rtimes A}(h\otimes a)\epsilon(h').
\]
Furthermore, also
\[
\begin{split}
\Delta((h\otimes a)* h')&=\Delta(d(h)a\cdot h')=\phi(|(d(h)a)_{2}|,|h'_{1}|)((d(h)a)_{1}\cdot h'_{1})\otimes((d(h)a)_{2}\cdot h'_{2})\\&=\phi(|h_{2}a_{2}|,|h'_{1}|)\phi(|h_{2}|,|a_{1}|)(d(h_{1})a_{1}\cdot h'_{1})\otimes(d(h_{2})a_{2}\cdot h'_{2})\\&=\phi(|(h\otimes a)_{2}|,|h'_{1}|)((h\otimes a)_{1}* h'_{1})((h\otimes a)_{2}* h'_{2})
\end{split}
\]
and 
\[
\begin{split}
(h\otimes a)*(h'h'')&=d(h)a\cdot h'h''=\phi(|(d(h)a)_{2}|,|h'|)((d(h)a)_{1}\cdot h')((d(h)a)_{2}\cdot h'')\\&=\phi(|h_{2}a_{2}|,|h'|)\phi(|h_{2}|,|a_{1}|)(d(h_{1})a_{1}\cdot h')(d(h_{2})a_{2}\cdot h'')\\&=\phi(|(h\otimes a)_{2}|,|h|)((h\otimes a)_{1}* h)((h\otimes a)_{2}* h')
\end{split}
\]
\end{invisible}
Hence we can define $H_{2}:=H\rtimes_{*}(H\rtimes A)$ and morphisms $d^{2}_{0},d^{2}_{1},d^{2}_{2}:H_{2}\to H_{1}$ in $\Hc$ as
\[
d^{2}_{0}(h\otimes h'\otimes a):=\epsilon(h)h'\otimes a,\ d^{2}_{1}(h\otimes h'\otimes a):=hh'\otimes a,\ d^{2}_{2}(h\otimes h'\otimes a):=h\otimes d(h')a 
\]
and morphisms $s^{2}_{0},s^{2}_{1}:H_{1}\to H_{2}$ as
\[
s^{2}_{0}(h\otimes a):=1_{H}\otimes h\otimes a,\ s^{2}_{1}(h\otimes a):=h\otimes1_{H}\otimes a.
\]
Observe that $d^{2}_{1}=m_{H}\otimes\mathrm{Id}_{A}$, which is clearly a morphism of graded coalgebras since $H\rtimes_{*}(H\rtimes A)$ has the tensor product coalgebra structure, is also a morphism of algebras. Indeed, given $h\otimes h'\otimes a$ and $k\otimes k'\otimes a'$ in $H\rtimes_{*}(H\rtimes A)$, we can compute 
\[
\begin{split}
    d^{2}_{1}&((h\otimes h'\otimes a)(k\otimes k'\otimes a'))=d^{2}_{1}(\phi(|(h'\otimes a)_{2}|,|k|)h((h'\otimes a)_{1}*k)\otimes(h'\otimes a)_{2}(k'\otimes a'))\\&=d^{2}_{1}(\phi(|h'_{2}\otimes a_{2}|,|k|)\phi(|h'_{2}|,|a_{1}|)h((h'_{1}\otimes a_{1})*k)\otimes(h'_{2}\otimes a_{2})(k'\otimes a'))\\&=d^{2}_{1}(\phi(|h'_{2}|,|a_{1}\cdot k|)\phi(|a_{2}\otimes a_{3}|,|k|)\phi(|a_{3}|,|k'|)h((d(h'_{1})a_{1})\cdot k)\otimes h'_{2}(a_{2}\cdot k')\otimes a_{3}a')\\&=\phi(|h'_{2}|,|a_{1}\cdot k|)\phi(|a_{2}\otimes a_{3}|,|k|)\phi(|a_{3}|,|k'|)h(d(h'_{1})\cdot(a_{1}\cdot k))h'_{2}(a_{2}\cdot k')\otimes a_{3}a'\\&\overset{\eqref{cm2}}{=}\phi(|h'_{2}\otimes h'_{3}|,|a_{1}\cdot k|)\phi(|a_{2}\otimes a_{3}|,|k|)\phi(|a_{3}|,|k'|)hh'_{1}(a_{1}\cdot k)S(h'_{2})h'_{3}(a_{2}\cdot k')\otimes a_{3}a'\\&=\phi(|h'_{2}|,|a_{1}\cdot k|)\phi(|a_{2}\otimes a_{3}|,|k|)\phi(|a_{3}|,|k'|)hh'_{1}(a_{1}\cdot k)\epsilon(h'_{2})(a_{2}\cdot k')\otimes a_{3}a'\\&=\phi(|a_{2}|,|k|)\phi(|a_{3}|,|kk'|)hh'(a_{1}\cdot k)(a_{2}\cdot k')\otimes a_{3}a'=\phi(|a_{2}|,|kk'|)hh'(a_{1}\cdot kk')\otimes a_{2}a'\\&=(hh'\otimes a)(kk'\otimes a')=d^{2}_{1}(h\otimes h'\otimes a)d^{2}_{1}(k\otimes k'\otimes a')
\end{split}
\]
and $d^{2}_{2}=\mathrm{Id}_{H}\otimes d^{1}_{1}$ is a morphism in $\Hc$. Moreover, the following relations 
\[
d^{2}_{0}\circ s^{2}_{0}=d^{2}_{1}\circ s^{2}_{0}=d^{2}_{1}\circ s^{2}_{1}=d^{2}_{2}\circ s^{2}_{1}=\mathrm{Id}_{H_{1}}
\]
are trivially satisfied as well as $s^{2}_{1}\circ s^{1}_{0}=s^{2}_{0}\circ s^{1}_{0}$. Furthermore, we have  
\[
d^{2}_{0}s^{2}_{1}(h\otimes a)=\epsilon(h)1_{H}\otimes a=s^{1}_{0}d^{1}_{0}(h\otimes a)\ \text{and}\ d^{2}_{2}s^{2}_{0}(h\otimes a)=1_{H}\otimes d(h)a=s^{1}_{0}d^{1}_{1}(h\otimes a) 
\]
for all $h\in H$ and $a\in A$. Finally, we have 
\[
d^{1}_{0}d^{2}_{1}(h\otimes h'\otimes a)=\epsilon(hh')a=d^{1}_{0}d^{2}_{0}(h\otimes h'\otimes a),\ d^{1}_{0}d^{2}_{2}(h\otimes h'\otimes a)=\epsilon(h)d(h')a=d^{1}_{1}d^{2}_{0}(h\otimes h'\otimes a)
\]
and $d^{1}_{1}d^{2}_{2}(h\otimes h'\otimes a)=d(hh')a=d^{1}_{1}d^{2}_{1}(h\otimes h'\otimes a)$ for all $h,h'\in H$ and $a\in A$.

Therefore, we have obtained a 2-truncated simplicial cocommutative color Hopf algebra and then, as said in Remark \ref{oss:2cosk}, we obtain a simplicial cocommutative color Hopf algebra $\mathcal{H}$ with first three components $H_{0}$, $H_{1}$ and $H_{2}$ defined as before. 

By definition $M(\mathcal{H})_{0}:=H_{0}=A$, $M(\mathcal{H})_{1}:=\mathrm{Hker}(d^{1}_{0})$ and $M(\mathcal{H})_{2}:=\mathrm{Hker}(d^{2}_{0})\cap\mathrm{Hker}(d^{2}_{1})$. Clearly, $\mathrm{Hker}(d^{1}_{0})=H\otimes\Bbbk1_{A}\cong H$ and we can show that $M(\mathcal{H})_{2}=\Bbbk(1_{H}\otimes1_{H}\otimes1_{A})$. As for $\mathrm{Hker}(d^{1}_{0})$, it is easy to show that $\mathrm{Hker}(d^{2}_{0})=H\otimes\Bbbk(1_{H}\otimes1_{A})$.

\begin{invisible}
An element $h\otimes h'\otimes a\in H\rtimes_{*}(H\rtimes A)$ is in $\mathrm{Hker}(d^{2}_{0})$ if
\[
\begin{split}
\phi(|h_{2}|,|h'_{1}|)\phi(|h_{2}|,|a_{1}|)\phi(|h'_{2}|,|a_{1}|)h_{1}\otimes h'_{1}\otimes a_{1}\otimes d^{2}_{0}(h_{2}\otimes h'_{2}\otimes a_{2})=h\otimes h'\otimes a\otimes 1_{H}\otimes 1_{A}
\end{split}
\]
but the first member is
\[
\phi(|h_{2}|,|h'_{1}|)\phi(|h_{2}|,|a_{1}|)\phi(|h'_{2}|,|a_{1}|)h_{1}\otimes h'_{1}\otimes a_{1}\otimes\epsilon(h_{2})h'_{2}\otimes a_{2}=\phi(|h'_{2}|,|a_{1}|)h\otimes h'_{1}\otimes a_{1}\otimes h'_{2}\otimes a_{2}
\]
and then we must have $\Delta_{H\rtimes A}(h'\otimes a)=h'\otimes a\otimes1_{H}\otimes 1_{A}$ from which $h'\otimes a=\epsilon(h'a)1_{H}\otimes 1_{A}$ and then $\mathrm{Hker}(d^{2}_{0})=H\otimes\Bbbk(1_{H}\otimes1_{A})$. 
\end{invisible}

But now an element $h\otimes1_{H}\otimes1_{A}$ is in $\mathrm{Hker}(d^{2}_{1})$ if 
\[
h_{1}\otimes1_{H}\otimes1_{A}\otimes d^{2}_{1}(h_{2}\otimes1_{H}\otimes1_{A})=h\otimes1_{H}\otimes1_{A}\otimes1_{H}\otimes1_{A},\ \text{i.e},\ h_{1}\otimes1_{H}\otimes1_{A}\otimes h_{2}\otimes1_{A}=h\otimes1_{H}\otimes1_{A}\otimes1_{H}\otimes1_{A} 
\]
and then $h=\epsilon(h)1_{H}$. Thus $M(\mathcal{H})_{2}=\Bbbk(1_{H}\otimes1_{H}\otimes1_{A})$. Moreover, by Corollary \ref{cor:simpobj}, we know that $M(\mathcal{H})_{i}=\Bbbk1$ for $i>3$ and $M(\mathcal{H})_{3}=\mathrm{Hker}(\partial_{2})$ with $\partial_{3}:\mathrm{Hker}(\partial_{2})\to M(\mathcal{H})_{2}$ given by the inclusion. Then $M(\mathcal{H})_{3}\cong\Bbbk1$ and so $\mathcal{H}$ has Moore complex with length one.
\end{proof}

Note that, in the previous proof, we need a 2-truncated simplicial cocommutative color Hopf algebra because, if we stop at the 1-truncation level and we apply the 1-coskeleton functor, we obtain a simplicial cocommutative color Hopf algebra with Moore complex of length two and there is no cancellation. 

Finally, denoting by $\mathrm{Simp}(\Hc)|_{l=1}$ the category of simplicial cocommutative color Hopf algebras with Moore complex of length one, we can show the foretold equivalence of categories:

\begin{proposition}
There is an equivalence between $\mathrm{HXMod}(\Hc)$ and $\mathrm{Simp}(\Hc)|_{l=1}$.
\end{proposition}

\begin{proof}
Using Proposition \ref{functor1} we obtain a functor $F:\mathrm{Simp}(\Hc)|_{l=1}\to\mathrm{HXMod}(\Hc)$ defined as
\[
\begin{tikzcd}
	(\mathcal{H}: && H_{3} & H_{2} & H_{1} & H_{0}
	\arrow["d_{0}"{description},from=1-5, to=1-6]
	\arrow[shift left=2, from=1-5, to=1-6, "d_{1}"]
	\arrow[shift left=2, from=1-6, to=1-5,"s_{0}"]
	\arrow["d_{0}"{description}, from=1-4, to=1-5]
	\arrow["d_{1}"{description}, shift left=3, from=1-4, to=1-5]
	\arrow[shift left=5, from=1-4, to=1-5,"d_{2}"]
	\arrow["s_{0}"{description}, shift left=3, from=1-5, to=1-4]
	\arrow[shift left=5, from=1-5, to=1-4, "s_{1}"]
	\arrow[shift left=2, from=1-3, to=1-4]
	\arrow[shift left=6, from=1-3, to=1-4]
	\arrow[shift left=4, from=1-3, to=1-4]
	\arrow[from=1-3, to=1-4]
	\arrow[shift left=2, from=1-4, to=1-3]
	\arrow[shift left=4, from=1-4, to=1-3]
	\arrow[shift left=6, from=1-4, to=1-3]
	\arrow[dotted, no head, from=1-1, to=1-3])
\end{tikzcd}\mapsto(H_{0},M(\mathcal{H})_{1}=\mathrm{Hker}(d^{1}_{0}),\partial_{1}:M(\mathcal{H})_{1}\to H_{0})
\]
and, using Proposition \ref{functor2}, we obtain a functor $G:\mathrm{HXMod}(\Hc)\to\mathrm{Simp}(\Hc)|_{l=1}$ defined as
\[
(A,H,d:H\to A)\mapsto\begin{tikzcd}
	(\mathcal{H}: && H\rtimes_{*}(H\rtimes A) & H\rtimes A & A
	\arrow[dotted, no head, from=1-1, to=1-3]
	\arrow["d_{0}"{description}, from=1-4, to=1-5]
	\arrow[shift left=2, from=1-4, to=1-5, "d_{1}"]
	\arrow[shift left=2, from=1-5, to=1-4, "s_{0}"]
	\arrow["d_{0}"{description}, from=1-3, to=1-4]
	\arrow["d_{1}"{description}, shift left=3, from=1-3, to=1-4]
	\arrow[shift left=5, from=1-3, to=1-4, "d_{2}"]
	\arrow["s_{0}"{description}, shift left=3, from=1-4, to=1-3]
	\arrow[shift left=5, from=1-4, to=1-3, "s_{1}"]).
\end{tikzcd}
\]
The functors $F$ and $G$ form an equivalence of categories. Indeed, we can compute
\[
(A,H,d:H\to A)\overset{G}\mapsto\begin{tikzcd}
	(\mathcal{H}: && H\rtimes_{*}(H\rtimes A) & H\rtimes A & A
	\arrow[dotted, no head, from=1-1, to=1-3]
	\arrow["d_{0}"{description}, from=1-4, to=1-5]
	\arrow[shift left=2, from=1-4, to=1-5, "d_{1}"]
	\arrow[shift left=2, from=1-5, to=1-4, "s_{0}"]
	\arrow["d_{0}"{description}, from=1-3, to=1-4]
	\arrow["d_{1}"{description}, shift left=3, from=1-3, to=1-4]
	\arrow[shift left=5, from=1-3, to=1-4, "d_{2}"]
	\arrow["s_{0}"{description}, shift left=3, from=1-4, to=1-3]
	\arrow[shift left=5, from=1-4, to=1-3, "s_{1}"])
\end{tikzcd}
\]
\[
\ \ \ \ \  \overset{F}\mapsto(A,M(\mathcal{H})_{1}=\mathrm{Hker}(d^{1}_{0}),\partial_{1}:M(\mathcal{H})_{1}\to A)
\]
and, as noted in the proof of Proposition \ref{functor2}, we have $M(\mathcal{H})_{1}=H\otimes\Bbbk1_{A}\cong H$ and $\partial_{1}=d^{1}_{1}|_{H\otimes\Bbbk1_{A}}=d$, hence clearly $FG\cong\mathrm{Id}$. Moreover, we can compute
\[
\begin{tikzcd}
	(\mathcal{H}: && H_{3} & H_{2} & H_{1} & H_{0}
	\arrow["d_{0}"{description},from=1-5, to=1-6]
	\arrow[shift left=2, from=1-5, to=1-6, "d_{1}"]
	\arrow[shift left=2, from=1-6, to=1-5,"s_{0}"]
	\arrow["d_{0}"{description}, from=1-4, to=1-5]
	\arrow["d_{1}"{description}, shift left=3, from=1-4, to=1-5]
	\arrow[shift left=5, from=1-4, to=1-5,"d_{2}"]
	\arrow["s_{0}"{description}, shift left=3, from=1-5, to=1-4]
	\arrow[shift left=5, from=1-5, to=1-4, "s_{1}"]
	\arrow[shift left=2, from=1-3, to=1-4]
	\arrow[shift left=6, from=1-3, to=1-4]
	\arrow[shift left=4, from=1-3, to=1-4]
	\arrow[from=1-3, to=1-4]
	\arrow[shift left=2, from=1-4, to=1-3]
	\arrow[shift left=4, from=1-4, to=1-3]
	\arrow[shift left=6, from=1-4, to=1-3]
	\arrow[dotted, no head, from=1-1, to=1-3])
\end{tikzcd}\overset{F}\mapsto(H_{0},M(\mathcal{H})_{1}=\mathrm{Hker}(d^{1}_{0}),\partial_{1}:M(\mathcal{H})_{1}\to H_{0})
\]
\[
\overset{G}\mapsto\begin{tikzcd}
	(\mathcal{H}': && (\mathrm{Hker}(d^{1}_{0})\rtimes_{*}(\mathrm{Hker}(d^{1}_{0})\rtimes H_{0})) & (\mathrm{Hker}(d^{1}_{0})\rtimes H_{0}) & H_{0}
	\arrow[dotted, no head, from=1-1, to=1-3]
	\arrow["d'_{0}"{description}, from=1-4, to=1-5]
	\arrow[shift left=3, from=1-4, to=1-5, "d'_{1}"]
	\arrow[shift left=3, from=1-5, to=1-4, "s'_{0}"]
	\arrow["d'_{0}"{description}, from=1-3, to=1-4]
	\arrow["d'_{1}"{description}, shift left=4, from=1-3, to=1-4]
	\arrow[shift left=6.5, from=1-3, to=1-4, "d'_{2}"]
	\arrow["s'_{0}"{description}, shift left=4, from=1-4, to=1-3]
	\arrow[shift left=6.5, from=1-4, to=1-3, "s'_{1}"])
\end{tikzcd}
\]
and both $\mathcal{H}'$ and $\mathcal{H}$ have Moore complex of length one, i.e., $M(\mathcal{H}')_{i}=M(\mathcal{H})_{i}=\Bbbk1$ for all $i\geq2$. Furthermore, $M(\mathcal{H}')_{0}=H_{0}=M(\mathcal{H})$ and $M(\mathcal{H}')_{1}=\mathrm{Hker}(d'^{1}_{0})=\mathrm{Hker}(d^{1}_{0})\otimes\Bbbk1_{H_{0}}\cong\mathrm{Hker}(d^{1}_{0})=M(\mathcal{H})_{1}$.
Hence $\mathcal{H}'$ and $\mathcal{H}$ are simplicial cocommutative color Hopf algebras which have isomorphic Moore complex, so they are isomorphic since the Moore functor $M:\mathrm{Simp}(\Hc)\to\mathrm{Ch}(\Hc)$ reflects isomorphisms, see \cite{bou2}.
\end{proof}

Note that, using Corollary \ref{corollary1}, we also obtain that $\mathrm{XMod}(\Hc)$ and $\mathrm{Simp}(\Hc)|_{l=1}$ are equivalent categories. This result can also be deduced from \cite[Proposition 4.4]{Bohm2}. Finally, we obtain:

\begin{corollary}
    The category $\mathrm{Simp}(\Hc)|_{l=1}$ is semi-abelian.
\end{corollary}

\begin{proof}
By Corollary \ref{corollary2} we know that both $\mathrm{XMod}(\Hc)$ and $\mathrm{HXMod}(\Hc)$ are semi-abelian, hence also $\mathrm{Simp}(\Hc)|_{l=1}$ is semi-abelian.
\end{proof}

\begin{remark}
    In \cite[Definitions 6.5 and 6.6]{Emir} the category of 2-crossed modules of cocommutative Hopf algebras is introduced and in \cite[Theorem 6.17]{Emir} it is shown that this category is equivalent to the category of simplicial cocommutative Hopf algebras with Moore complex of length 2. It would be interesting to extend this result to the category $\Hc$ introducing the notion of 2-color Hopf crossed module. We leave this as a possible future project.
    
    Moreover, as it is said in \cite[page 301]{Sterck}, the notions of 2-crossed module and Hopf crossed square for the category of cocommutative Hopf algebras are not equivalent, hence it is reasonable to think that 2-color Hopf crossed modules and color Hopf crossed squares could be two not equivalent generalizations.
\end{remark}

\noindent\textbf{Acknowledgements}. This paper was written while the author was member of the “National Group for Algebraic and Geometric Structures and their Applications” (GNSAGA-INdAM). He was partially supported by the Ministry for University and Research (MUR) within the National Research Project (PRIN 2022) "Structures for Quivers, Algebras and Representations (SQUARE)". The author would like to thank A. Ardizzoni and A. S. Cigoli for the careful reading of this paper and for meaningful comments.

\end{document}